\documentclass[12pt,reqno]{amsart}
  
\newcommand{\cs}{$\clubsuit$} 
 \usepackage{comment}
\usepackage{tikz}
\usepackage{latexsym}
\usepackage{graphicx}
\usepackage{amssymb}
 \graphicspath{{Figures/}}
 \usepackage[show]{ed}

\newcommand{\red}[1]{{\color{red}{#1}}}
\newcommand{\rankc}{\operatorname{rank}_{\mathbb{C}}}
\newcommand{\rankr}{\operatorname{rank}_{\mathbb{R}}}

\renewcommand{\Re}{{\operatorname{Re}\,}}
\renewcommand{\Im}{{\operatorname{Im}\,}}

\renewcommand{\epsilon}{\varepsilon}

\newcommand{\R}{{\mathbb R}}
\newcommand{\C}{{\mathbb C}}

\newcommand{\Z}{{\mathbb Z}}

\newcommand{\lan}{\left\langle}
\newcommand{\ran}{\right\rangle}
\newcommand{\mc}[1]{\mathcal{#1}}
\newcommand{\e}{\epsilon}

\newcommand{\re}{\mathbb{R}}
\newcommand{\Cc}{C_c^\infty}
\newcommand{\Id}{\operatorname{Id}}

\newcommand{\h}{\hbar}

\newcommand{\supp}{{\operatorname{supp\,}}}
\newcommand{\loc}{\operatorname{loc}}
\renewcommand{\phi}{\varphi}

\newcommand{\ical}{\mathfrak{I}}

\newcommand{\rcal}{\mathfrak{R}}

\newtheorem{theo}{{\sc Theorem}}

\newtheorem{cor}{{\sc Corollary}}[section]

\newtheorem{lem}[cor]{{\sc Lemma}}

\newtheorem{prop}[cor]{{\sc Proposition}}
\numberwithin{equation}{section}

\newenvironment{rem}{\medskip\noindent{\it Remark:\/} }{\medskip}

\newtheorem{defn}[cor]{{\sc Definition}}

\usepackage{hyperref}

\title[Sup bounds in QCI Systems]{Pointwise bounds for joint eigenfunctions of quantum completely integrable systems }

\author{Jeffrey Galkowski}
\address{Department of Mathematics, Northeastern University, Boston, MA, USA}
\email{jeffrey.galkowski@northeastern.edu}

\author{ John A. Toth}
\address{Department of Mathematics and Statistics, McGill University, Montr\'eal, QC, Canada}
\email{john.toth@mcgill.ca}
\date{}

\begin{document}

\maketitle
\begin{abstract}
Let $(M,g)$ be a compact Riemannian manifold and $P_1:=-h^2\Delta_g+V(x)-E_1$ so that $dp_1\neq 0$ on $p_1=0$. We assume that $P_1$ is quantum completely integrable in the sense that there exist functionally independent pseuodifferential operators $P_2,\dots P_n$  with $[P_i,P_j]=0$, $i,j=1,\dots n$. We study the pointwise bounds for the joint eigenfunctions, $u_h$ of the system $\{P_i\}_{i=1}^n$ with  $P_1u_h=E_1u_h+o(1)$. In Theorem \ref{QCI}, we first give polynomial improvements over the standard H\"ormander bounds for typical points in $M$. In two and three dimensions, these estimates agree with the Hardy exponent $h^{-\frac{1-n}{4}}$ and in higher dimensions we obtain a gain of $h^{\frac{1}{2}}$ over the H\"ormander bound.

In our second main result (Theorem \ref{expdecay}), under a real-analyticity assumption on the QCI system, we give exponential decay estimates for joint eigenfunctions at points outside the projection of invariant Lagrangian tori; that is  at points $x\in M$ in the ``microlocally forbidden" region  $p_1^{-1}(E_1)\cap \dots \cap p_n^{-1}(E_n)\cap T^*_xM=\emptyset.$  These bounds are sharp locally near the projection of the invariant tori.
\end{abstract}

\section{Introduction}

Let $(M^n,g)$ be a closed, compact $C^{\infty}$ manifold and $P_1(h): C^{\infty}(M) \rightarrow C^{\infty}(M)$ a self-adjoint semiclassical pseudodifferential  operator of order $m$ that is elliptic in the classical sense, i.e. $|p_1(x,\xi)|\geq c|\xi|^m.-C.$  Here, $h $ takes values in a discrete sequence $(h_j)_{j=1}^{\infty}$ with $h_j \rightarrow 0^+$ as $j \rightarrow \infty.$ We assume in addition that there exist functionally independent $h$-pseudodifferential operators $P_2(h),...,P_{n}(h)$ with the property that 
\begin{equation} \label{commutator}
[ P_{i}(h), P_{j}(h) ] = 0; \,\,\, i,j = 1,...,n. \end{equation}
In that case we say that $P_1(h)$ is quantum completely integrable (QCI).
  Given the joint eigenvalues $ E(h) = (E_{1}(h),...,E_n(h)) \in \R^n$ of $P_1(h),...P_n(h)$ we denote an $L^2$-normalized joint eigenfunction with joint eigenvalue $E(h)$  by $u_{E,h}$ (here, for notational simplicity we drop the dependence of $E$ on $h$ in the notation) and consequently,
$$  P_{j}(h) u_{E,h} = E_{j}(h)  u_{E,h}.$$
When the joint energy value $E$ is understood, we will sometimes abuse notation and simply write $u_h = u_{E,h}.$

The associated classical integrable system is governed by the moment map
 \begin{equation} \label{moment map}
{\mathcal P}: = (p_1,...,p_n):T^*M \rightarrow \R^n \end{equation}
 where $p_j \in C^{\infty}(T^*M); j=1,...,n$ are the semiclassical principal symbols of $P_j(h); j=1,...,n.$  For convenience, we will denote the corresponding QCI system by $\hat{\mathcal P} := (P_1,...,P_n).$
 
We assume throughout that the classical integrable system $p$ is {\em Liouville integrable}; that is there exists an open dense subset $ T^*M_{reg} \subset T^*M $  such that 
 \begin{equation} \label{liouville}
 \text{rank} (dp_1(x,\xi),....,dp_n(x,\xi) ) = n  \,\, \forall \, (x,\xi) \in T^*M_{reg}.\end{equation}
Following the notation in \cite{TZ}, we let ${\mathcal B} := {\mathcal P}(T^*M)$ and ${\mathcal B}_{reg} ={\mathcal P}(T^*M_{reg})$ denotes the set of regular values of the moment map. 

Since ${\mathcal P}$ is proper, the Liouville-Arnold theorem determines the symplectic structure of the level sets ${\mathcal P}^{-1}(E)$ where $E \in {\mathcal B}_{reg}.$ The level set
\begin{equation} \label{Liouville}
{\mathcal P}^{-1}(E) = \cup_{k=1}^{M} \Lambda_k(E), \end{equation}
where the $\Lambda_k(E)$'s are Lagrangian tori which are invariant under the joint bicharacteristic flow $G^{t}: T^*M \rightarrow T^*M, \, t = (t_1,...,t_n) \in \R^n,$
$ G^t (x,\xi) = \exp t_{1} H_{p_1} \circ \cdots \circ \exp t_n H_{p_n}(x,\xi).$ Here, $H_{p_j} = \sum_{k} \partial_{\xi_k }p_j  \partial_{x_k} - \partial_{x_k } p_j \partial_{\xi_k}$ is the Hamilton vector field of $p_j.$\\

In this paper, we are concerned with two questions regarding the joint eigenfunctions: (i) eigenfunction supremum bounds and (ii) eigenfunction decay estimates in the microlocally forbidden region, $M \setminus \pi( {\mathcal P}^{-1}(E) ).$

\subsection{Supremum Estimates}
  To state our first result on sup bounds, we need a definition. 

\begin{defn} \label{morseassumption}
Let $(M^n,g)$ be a Riemannian manifold and $P_j(h); j=1,...,n$ be a non-degenerate, QCI system with Hamiltonian $\hat{H} = P_1(h).$  Suppose  $E_1$ satisfies $\partial_\xi p_1\neq 0$ on $p_1^{-1}(E_1)$ and set 
$$\Sigma_{x,E_1}:=  \{ \xi \in T_x^*M;  \, p_1(x,\xi)= E_1\}.$$
We say that the system is of {\em Morse type at $x \in M$} if 
there exists $f \in C^{\infty}(\R^n,\R)$ and an $h$-pseudodifferential operator $ Q(h):= f(P_1(h),...,P_n(h))$ with the property that its principal symbol
$$ q  \, |_{ \Sigma_{x,E_1} } \,\,\text{is Morse for all} \,\, x \in M. $$
\end{defn}

Our first main result is

\begin{theo}
\label{t1}
 \label{QCI} Let $(M^n,g)$ be compact Riemannian manifold and  $\hat{\mathcal P}$ be a QCI system with quantum Hamiltonian  $P_1(h)= - h^2 \Delta_g + V $ where $V \in C^{\infty}(M;\R)$ and  $E_1 \in \R$ is a regular value of $p_1,$  {i.e. so that $dp_1 |_{p_1^{-1}(E_1)}  \neq 0.$}  Suppose  $\Omega$ is an open set with $\overline{\Omega} \subset  \{ V < E_1 \}$ and that the system $\hat{\mathcal P}$ is Morse type at $x$ for all $x \in \overline{\Omega}.$ Then, the $L^2$-normalized joint eigenfunctions, $u_h$, with $P_1(h) u_h = E_1(h) u_h, \,\,\, E_1(h) = E_1 +o(1)$ satisfy the supremum bounds

\begin{equation} \label{anydim}
\| u_h \|_{L^\infty(\overline{\Omega})} = O( h^{(2-n)/2 }), \quad  n>3. \end{equation}.

In the cases where $n=2$ or $n=3,$ one gets the Hardy-type supremum bounds: 
\begin{equation}
\label{local}
  \| u_h \|_{L^\infty( \overline{\Omega})} = \begin{cases}O(h^{-1/4})&n=2   \\ 
  O(h^{-1/2} |\log h|^{1/2} ), &n=3.\end{cases}\end{equation}

\end{theo} 

\pagebreak

\begin{rem} 
 \begin{enumerate}[(i)]
 \item In the special case of Laplace eigenfunctions, $P_1(h) = -h^2 \Delta_g - 1$; that is, $V=0$ and $E_1 = 1.$ 
 
 \smallskip
\item The estimate (\ref{anydim}) in Theorem \ref{QCI} gives an explicit polynomial improvement over the well-known H\"{o}rmander bound $ \|u_h \|_{L^\infty} = O(h^{ (1-n)/2}).$  In dimensions $n=2,3,$  modulo the logarithmic factor in the $n=3$ case, both the estimates in (\ref{local}) are consistent with the {\em Hardy type} bound $\| u_h \|_{L^\infty}  = O(h^{(1-n)/4})$.  
Moreover, these estimates are sharp and are also quite robust in that they apply to many QCI examples either {\em globally} (e.g. Liouville Laplacians or Neumann oscillators on tori), or {\em locally} away from isolated points (e.g. Laplacians on  convex surfaces of revolution, Laplacians on asymmetric ellipsoids (n=2,3), quantum Neumann oscillators (n=2,3), quantum spherical pendulum, and quantum Euler and Kovalevsky tops). We describe how the above results apply explicity in several classical examples in section \ref{examples}.

 In the global cases, the bounds in Theorem \ref{QCI} holds for {\em all} $\Omega$ with $\overline{\Omega} \subset \{ V < E \}.$ Otherwise, one must delete arbitrarily small  (but fixed independent of $h$) balls centered at a finite number of points (e.g. the umbilic points of an triaxial ellipsoid, or the poles of an  convex surface of revolution.)
Finally, we point out in the case of the Laplacian, $V=0,$ so that the potential well is the entire manifold, $M$, and the corresponding sup bounds hold over all of $M;$ that is, one can set $\overline{\Omega} = M$ in (\ref{local}).

\smallskip
\item We point out that in Theorem~\ref{QCI} we fix only the energy $E_1$. In particular, it is a statement about \emph{all} joint eigenfunctions so that $P_1u_h=(E_1+o(1))u_h$ and we crucially do not require that the total energy, $E\in\mathcal{B}$ is regular i.e. we do not require $E\in \mathcal{B}_{\text{reg}}$.
\end{enumerate}

\end{rem}

One of the quantum integrable examples where the Morse hypothesis of Theorem~\ref{t1} is \emph{not} satisfied at every point is that of the triaxial ellipsoid
\begin{equation}
\label{e:ellipsoid}
\mc{E}:=\Big\{ w\in \R^3\,\big|\, \sum_{j=1}^3\frac{w_j^2}{a_j^2}=1,\,0<a_3<a_2<a_1\Big\}.
\end{equation}
Here, there are four exceptional points, $\{p_j\}_{j=1}^4\in \mc{E}$, the umbillic points, where the integrable system is not of Morse type. Combining the proof of Theorem~\ref{t1} with results from~\cite{CG18}, we prove the following sup bound for the joint eigenfunctions:
\begin{theo}
\label{t:ellipse}
Let $\mc{E}$ as in~\eqref{e:ellipsoid} and $P=-h^2\Delta_g-1$. Then there is $C>0$ so that any $L^2$ normalized joint eigenfunction, $u_h$ of the QCI system satisfies
$$
\|u_h\|_{L^\infty(\mc{E})}\leq Ch^{-\frac{1}{2}}|\log h|^{-\frac{1}{2}}.
$$
\end{theo}
In~\cite{To96}, the second author showed that there are constants $c,h_0>0$ and a sequence of $L^2$ normalized joint eigenfunctions of the QCI system satisfying
$$
|u_h(p_i)|\geq ch^{-\frac{1}{2}}|\log h|^{-\frac{1}{2}}, \quad 0<h<h_0,
$$
and consequently, the estimate in Theorem \ref{t:ellipse} is sharp.


\subsection{Comparison with previous $L^\infty$ estimates}
In general, for normalized Laplace eigenfunctions on a compact manifold $M$ of dimension $n$ i.e. solving $(-h^2\Delta_g-1)u=0$, the celebrated works~\cite{Ho68,Ava,Lev} show that 
\begin{equation}
\label{e:stdBound}
\|u_h\|_{L^\infty}\leq Ch^{\frac{1-n}{2}}.
\end{equation}
Under certain geometric conditions on the manifold $M$, this bound can be improved to 
\begin{equation}
\label{e:stdBound1}
\|u_h\|_{L^\infty}=o(h^{\frac{1-n}{2}}).
\end{equation}
These conditions include non-existence of recurrent points (see~\cite{STZ,Gdefect,CG17}), which in particular is satisfied for manifolds without conjugate points. Under a certain uniform version of the non-recurrent hypothesis~\cite{CG18} shows that this can be improved to 
\begin{equation}
\label{e:stdBound2}
\|u_h\|_{L^\infty}\leq C\frac{h^{\frac{1-n}{2}}}{\sqrt{\log h^{-1}}}.
\end{equation}
This non-recurrent hypothesis is in particular satisfied on manifolds without conjugate points where improved $L^\infty$ estimates have been proved using the Hadamard parametrix in~\cite{Berard77,Bon}. Finally, in forthcoming work~\cite{GT18}, the authors give improvements of the form
\begin{equation}
\label{e:stdBound3}
\|u_h\|_{L^\infty}\leq Ch^{\frac{1-n}{2}+\delta} 
\end{equation}
for some explicit $\delta>0$ when the manifold has integrable geodesic flow. The only other polynomial improvements that the authors are aware of occur in the case of Hecke--Maas forms on certain arithmetic surfaces~\cite{I-S}.

In this paper, we assume that eigenfunctions are joint eigenfunctions of a quantum complete \emph{system} of equations. In~\cite{TZ02}, it is shown that if QCI Laplace eigenfunctions have sup-norms that are $O(1)$, then the manifold is, in fact, flat. Therefore, it is natural to understand the $L^\infty$ growth of eigenfunctions in the QCI case. We note that the QCI assumption is very rigid and allows us to give much stronger than the results mentioned above. Indeed, Theorem~\ref{QCI} achieves the so-called Hardy estimate in dimension $n=2$, and $n=3$ (modulo a $\sqrt{\log h^{-1}}$ loss)
$$ 
\|u_h\|_{L^\infty}\leq Ch^{-\frac{1-n}{4}} 
$$
which is expected to hold at a generic point on a generic manifold. Moreover, in any dimension $n$, under a generic assumption on the QCI system, we are able to give an explicit polynomial improvement over~\eqref{e:stdBound}. 

While this is a dramatic improvement over the bounds above, it is important to note that the assumption of quantum complete integrability is highly sensitive. First, any small perturbation of the original operator (even a lower order perturbation) will destroy the property of being quantum integrable. Furthermore, even if the Laplacian is quantum integrable, it is not clear that all eigenfunctions for the Laplacian are joint eigenfunctions of the corresponding QCI system.  On the other hand, the approaches used to obtain~\eqref{e:stdBound},~\eqref{e:stdBound1},~\eqref{e:stdBound2} and ~\eqref{e:stdBound3} are robust to lower order perturbations and apply to \emph{all} sequences of eigenfunctions.

Our bounds are related to those in~\cite{Sa} where Sarnak shows that on a locally symmetric space of rank $r$, 
$$
\|u_h\|_{L^\infty}\leq Ch^{\frac{r-n}{2}}.
$$
and the generalization of this bound to joint quasimodes of $r$ essentially commuting operators with independent fiber differentials~\cite{Ta18}. We point out that while for some specific energy levels $E$, there are points satisfying the independent fiber differential assumption, the only quantum integrable example we are aware of in which there is a \emph{single} point $x$ satisfying this assumption for all energy levels is that of the flat torus. 
We also note that our results in Theorem \ref{QCI} apply in the case of many QCI systems that {\em do not} arise from isometric group actions; these include Liouville Laplacians on tori, Laplacians on asymmertric ellipsoids, quantum Neumann oscillators on spheres and quantum Kowalevsky tops, among others.

\subsection{Exponential Decay Estimates}

Our next result deals with exponential decay estimates for joint eigenfunctions in the  microlocal ``forbidden" region $M \setminus \pi(\Lambda_{\R})$ with 
$$
\Lambda_{\R}=\bigcap_{i=1}^n p_i^{-1}(E_i).
$$ 
We make the additional assumption that $P_j(h): j=1,..,n$ are real-analytic, $h$-differential operators and that the restricted canonical projection
$$ \pi_{\Lambda}: \Lambda_{\R}(E) \to M, \quad E= (E_1,...,E_n),$$
has a fold singularity along the {\em caustic} ${\mathcal C}_{\Lambda} = \pi_{\Lambda}^{-1} ( \, \partial \pi_{\Lambda}(\Lambda_{\R}(E)) \, ).$   One can complexify $\Lambda_{\R}$ to a complex submanifold, $\tilde{\Lambda},$  of the complexification, $\widetilde{T^*M}$, of the real cotangent bundle. Here, $\tilde{\Lambda}$ is Lagrangian with respect to the canonical complex symplectic form $\Omega^{\C} = d \omega^{\C}$ on $\widetilde{T^*M}$, where $\omega^{\C}$ is the complex canonical one-form on $\widetilde{T^*M}.$ In the terminology of~\cite{SjA}, $\tilde{\Lambda}$  is $\C$-Lagrangian. There is a further submanifold $\tilde{\Gamma}_I \subset \tilde{\Lambda}$ given by
$$\tilde{\Gamma}_I:= \tilde{\Lambda} \cap  \widetilde{T^*M}_M$$
that is of particular interest to the study of eigenfunction decay.  Roughly speaking, $\tilde{\Gamma}_I$ is subset of $\tilde{\Lambda}$ that consists of points with real base coordinates.  We also show in subsection \ref{s:fold} (see Proposition \ref{ilag}), under the fold assumption, one can characterize the structure of $\tilde{\Gamma}_I $ quite readily near ${\mathcal C}_{\Lambda};$ at least locally, one can write
$$ \tilde{\Gamma}_I = \Lambda_{\R} \cup \Gamma_I.$$
Both $\Lambda_{\R}$ and $\Gamma_I$ are isotropic with respect to $\Im \Omega^{\C}$ (ie. they are $I$-isotropic) and $\Gamma_I$ locally projects to the microlocally forbidden region, $M \setminus \pi(\Lambda_{\R}).$  Moreover, $\Gamma_I$ is locally a graph over $M$ away from the projection of the caustic $\partial \pi(\Lambda_\R)$  with
\begin{equation} \label{complexgraph}
 \Gamma_I = \{ (x, d_x \psi(x)); x \in \pi (\Gamma_I) \} \end{equation} where $\psi$ is complex-valued and real-analytic. In addition, as a consequence of the fold assumption, $\Gamma_I$ can be further decomposed as a union over two branches $\Gamma_I^+ \cup \Gamma_I^-,$ where these branches are (locally) characterized as follows: given any local smooth curve $\gamma^\pm(\alpha_0, \alpha) \subset \Gamma_I^{\pm}$ joining $\alpha_0 \in {\mathcal C}_{\Lambda}$ to $\alpha \in \Gamma_I^{\pm},$
$$ \pm  \int_{\gamma^{\pm} (\alpha_0,\alpha)} \Im \omega^{\C}  \geq 0.$$

 In view of (\ref{complexgraph}), there exist  locally well-defined functions $S^\pm: \pi(\Gamma_I^\pm) \to \C$ that are real-analytic away from $\partial \pi (\Lambda_{\R})$ with
  $$S^+(x) =  \int_{\gamma^+} \Im \omega^{\C}, \,\,\, \alpha= (x, d_x \psi(x)).$$
   We then define the {\em complex action function} locally to be 
 $$ S(x):=  \psi^+(x) \geq 0; \quad x \in \pi(\Gamma_I^-).$$

Our main result on the exponential decay of joint eigenfunctions is:

\begin{theo} \label{expdecay}
Suppose that $P(h) = (P_1(h),...,P_n(h))$ is  a QCI system of real-analytic, jointly elliptic, $h$-differential operators and $E \in {\mathcal P}(T^*M)$ a regular level  of the moment map.
Suppose, in addition, that the caustic ${\mathcal C}_{\Lambda}$ is a fold. Then, there exists an $h$-indepedent neighbourhood, $V \supset \pi(\Lambda_{\R}),$ such that for any open $\Omega \Subset ( \,V \setminus \pi(\Lambda_{\R}) \, )$ and any $\epsilon >0,$ { there exists $h_0(\epsilon,\Omega)>0$ such that for $h \in (0,h_0(\epsilon,\Omega)],$} and $u_{h}$ a joint eigenfunction of $P(h)$ with energy $E$,
$$ \sup_{x \in \Omega} | e^{ (1-\epsilon) S(x)/h} \, u_h(x) | = O_{\epsilon} (e^{\beta(\epsilon)/h}),$$
where $\beta(\epsilon) = O(\epsilon^{1/2})$ as $\epsilon \to 0^+.$
\end{theo}

As we show in section \ref{examples}, under the real-analyticity assumption the decay estimate in Theorem \ref{expdecay} is sharp and improves on results of the second author in \cite{To98}.
Moreover, the fold assumption is satisfied for generic joint energy levels when $n \geq 2.$  In the cases where there exist appropriate coordinates in terms of which the classical generating function is separable, one can show  that the decay estimates in Theorem \ref{expdecay} are still satisfied for non-generic energy levels $E \in {\mathcal B}_{reg}$. The latter condition is satisfied in all cases that we know of (see remark \ref{nonfold} for more details)

\medskip
\noindent {\sc Acknowledgements.} 
J.G. is grateful to the National Science Foundation for support under the Mathematical Sciences Postdoctoral Research Fellowship  DMS-1502661. J.T. was partially supported by NSERC Discovery Grant \# OGP0170280 and  by the French National Research Agency project Gerasic-ANR-
13-BS01-0007-0.

\section{Sup bounds for QCI eigenfunctions: proof of Theorem \ref{QCI}}

\begin{proof} We assume first that $n=2$ and that $P_1(h) = - h^2 \Delta_g$, {$E_1=1$} and indicate the minor changes in the case where $P_1(h) = - h^2 \Delta_g + V(x),$ at the end.  Since we assume the QCI condition, instead of working with long-time propagators, it simplifies the analysis to use small-time joint propagators. We will also assume without loss of generality that $E_1=0$ (replacing $P_1$ by $P_1-1$). Suppose $P_1(h) u_h =0$ and with $Q(h):= p_2^w(h) - E(h)$ we have $Q(h) u_h = 0.$ As usual, we let $\rho \in S(\R)$ with $\rho(0) = 1$ and with $\epsilon >0$ small we choose supp $\, \hat{\rho} \subset [\epsilon, 2 \epsilon].$ 

 Then,  since $[P_1,Q] = 0,$ for any $x \in M,$ we can write 

$$  u_h(x) = \int_{\R} \int_{\R} \Big(  e^{i t P_1(h)/h} e^{is Q(h)/h}  \, u_h \Big) \, \hat{\rho}(t) \, \hat{\rho}_1(s) \, ds dt$$
Let $\chi \in C^{\infty}_0(\R; [0,1])$ with $\chi \equiv 1$ on $[-\e,\e]$ and $\supp \chi\subset[-2\e,2\e]$ and set $\chi(h) = \chi(P_1(h)).$ Since
 $$(1-\chi(h)) u_h = 0$$
and by construction $[\chi, P_1] = 0$ and $[\chi, Q] = 0,$ we can $h$-microlocalize the identity above and write
\begin{equation} \label{qci1}
u_h(x) = \int_{\R} \int_{\R} \Big(  e^{i t P_1(h)/h} \chi(h) e^{is Q(h)/h} \chi(h)  \, u_h \Big) \, \hat{\rho}(t) \, \hat{\rho}(s) \, dt ds + O(h^{\infty}). \end{equation}
By a standard stationary phase argument (see e.g.~\cite[Section 3.1]{GT},~\cite[Theorem 4]{BGT},~\cite[Lemma 5.1.3]{SoggeBook}),  we can write the Schwartz kernelof $\int_{\R} \hat{\rho}(t) e^{itP_1(h)/h} \chi(h) \, dt$ in the form

\begin{equation}\label{sogge}
K_1(x,y,h) = (2\pi h)^{\frac{1-n}{2}} e^{ir(x,z)/h} \hat{\rho}(r(x,y)) a(x,y,h) + O_{C^{\infty}}(h^\infty) \end{equation}

where $a(x,y,h) \sim \sum_{j=0}^\infty a_j(x,y) h^j, \,\,\, a_j \in C^{\infty}$ and $r(\cdot,\cdot)$ denotes geodesic distance in the metric $g.$ Thus, letting $r_{inj}=\textup{inj}(M)$ and choosing geodesic normal coordinates, $y: B_{r_{inj}}(x) \to \R^n$ centered at $x \in M,$ we have that the phase 
$$r(x,y) = |x-y|.$$
The microlocalized propagator, $U(s;h) := e^{is Q(h)/h} \chi(h) $ has Schwartz kernel that is an $h$-FIO of the form
\begin{equation} \label{fio1}
U(s,y,z;h) = (2\pi h)^{-n} \int_{\R^n} e^{i [ S(s,y,\eta) - \langle z, \eta \rangle ]/h} \, b(s,y,z,\eta;h) \, d\eta + O_{C^{\infty}}(h^{\infty}),
\end{equation}
where $a \in S^0$ with $b \sim_{h \to 0^+} \sum_{j=0}^{\infty} b_j h^{j}$ and where $S(s,y,\eta)$ solves the eikonal equation
$$ \partial_sS = q(y,\partial_y S), \quad S(0,z,\eta) = \langle z, \eta \rangle.$$
%

Then, in view of (\ref{sogge}) and (\ref{fio1}), and  with 

$$K(x,z):= \Big( \, \int e^{itP_1/h} \chi(h) e^{isQ/h} (h) \hat{\rho}(t)\hat{\rho}(s)dsdt \, \Big)(x,z),$$
we have that
\begin{equation} \label{compositekernel}
K(x,z)= (2\pi h)^{\frac{1-n}{2}-n}\int e^{\frac{i}{h}(|x-y|+S(s,y,\eta)-\langle z,\eta\rangle)}\hat{\rho}(|x-y|)c(x,y,h)\hat{\rho}(s)dsdyd\eta
\end{equation}
where, $c(x,z,h) \sim \sum_{j=0}^{\infty} c_j(x,z) h^j.$ and
$$
\partial_s S(s,y,\eta)=q(y,\partial_y S(s,y,\eta)),\qquad S(0,y,\eta)=\langle y,\eta\rangle.
$$
Performing stationary phase in $(y,\eta)$ gives that at the critical point $(y_c(x,z,s),\eta_c(x,z,s))$,
$$
\begin{gathered}
\frac{y_c-x}{|y_c-x|}+\partial_yS(s,y_c,\eta_c)=0\\
\partial_\eta S(s,y_c,\eta_c)-z=0
\end{gathered}
$$
Let
$$
\Phi(x,z,s)=|x-y_c(x,z,s)|+S(s,y_c(x,z,s),\eta_c(x,z,s))-\langle z,\eta_c(x,z,s)\rangle
$$
so that 
$$
K(x,z)= (2\pi h)^{\frac{1-n}{2}}\int e^{\frac{i}{h}\Phi(x,z,s)}\tilde{c}(x,z,s)ds.
$$
Then, by Cauchy--Schwarz,
\begin{align*}
|u_h(x)|^2&=\Big|\int e^{\frac{i}{h}\Phi(x,z,s)}\tilde{c}(x,z,s) u_h(z) \, ds dz\Big|^2\\
&\leq \Big( \int \Big| \int e^{\frac{i}{h}\Phi(x,z,s)}\tilde{c}(x,z,s)ds\Big|^2 dz \Big) \cdot \|u_h\|^2_{L^2}.
\end{align*}
Now, we observe that 
$$
(2\pi h)^{1-n}\int \Big| \int e^{\frac{i}{h}\Phi(x,z,s)}\tilde{c}(x,z,s)ds\Big|^2dz =(2\pi h)^{1-n}\int e^{\frac{i}{h}(\Phi(x,z,s)-\Phi(x,z,t))}\tilde{c}(x,z,s)\overline{\tilde{c}(x,z,t)}dsdtdz
$$
and also note that 
$$y_c(x,z,0)=z,\qquad \eta_c(x,z,0)=\frac{x-z}{|x-z|}$$
and compute

\begin{align*}
\partial_s\Phi&=\frac{\langle x-y_c,-\partial_{s}y_c\rangle}{|x-y_c|}+\partial_sS+\langle \partial_yS,\partial_s y_c\rangle+\langle \partial_\eta S,\partial_s \eta_c\rangle-\langle z,\partial_s\eta_c\rangle\\
&=\frac{\langle x-y_c,-\partial_{s}y_c\rangle}{|x-y_c|}+q(y_c,\partial_{y}S)+\frac{\langle x-y_c,\partial_s y_c\rangle}{|x-y_c|}+\langle z,\partial_s \eta_c\rangle-\langle z,\partial_s\eta_c\rangle\\
&=q\Big(y_c,\frac{x-y_c}{|x-y_c|}\Big)
\end{align*}
Therefore,
$$\Phi(x,z,s) =\int_0^s q\Big(y_c(x,z,r),\frac{x-y_c(x,z,r)}{|x-y_c(x,z,r)|}\Big)dr + q(z,\frac{x-z}{|x-z|}).$$
and
$$
\Phi(x,z,s)-\Phi(x,z,t)=\int_{t}^sq\Big(y_c(x,z,r),\frac{x-y_c(x,z,r)}{|x-y_c(x,z,r)|}\Big)dr.
$$
In particular, 
$$
\Phi(x,z,s)-\Phi(x,z,t)=(s-t)q(z,\frac{x-z}{|x-z|})+(s^2f(x,z,s)-t^2f(x,z,t))
$$
Therefore, changing variables to $S=t-s$ $T=t+s$, 
\begin{equation} \label{integral}
|u_h(x)|^2\leq \|u_h\|^2  \cdot  (2\pi h)^{1-n} \int e^{\frac{iS}{h} \big[ q(z,\frac{x-z}{|x-z|})+O_{C^\infty}(T) \big]}c_1(x,z,S,T)dSdTdz.
\end{equation}

We split the integral into two pieces
$$
(2\pi h)^{1 - n}\int e^{\frac{iS}{h}(q(z,\frac{x-z}{|x-z|})+O_{C^\infty}(T))}\chi(Sh^{-1})c_1(x,z,S,T)dSdTdz\leq Ch^{2-n}
$$
and 
\begin{equation} \label{second}
(2\pi h)^{1- n}\int e^{\frac{iS}{h}(q(z,\frac{x-z}{|x-z|})+O_{C^\infty}(T))}(1-\chi(Sh^{-1}))c_1(x,z,S,T)dSdTdz.
\end{equation}
  First, note that since $H_pq=0$, $q(z,\frac{x-z}{|x-z|})=q(x,\frac{x-z}{|x-z|})$. Therefore,  the Morse assumption on $q|_{S^*_xM}$ allows us to perform stationary phase in $z$ with $hS^{-1}$ as a small parameter in the second integral  (\ref{second}). The result is that the latter integral is 
$$
\leq C h^{1-n} h^{(n-1)/2} \int |S^{(1-n)/2}(1-\chi(Sh^{-1}))\chi(T)|dSdT\leq C h^{(1-n)/2} \, \int_{h}^{1} S^{(1-n)/2} \, dS.
$$

Summarizing, we have proved that
\begin{equation} \label{upshot1}
\begin{aligned}
|u_h(x)|^2 &\leq C \, h^{1-n} \, \Big(   h^{\frac{n-1}{2}}\int_{h}^{1} S^{(1-n)/2} \, dS  + h \Big)\\
&\leq \left\{\begin{aligned}h^{\frac{1}{2}}&&n=2\\h^{-1}\log h^{-1}&&n=3\\h^{2-n}&&n> 3\end{aligned}\right.
\end{aligned} \end{equation}
Taking square roots completes the proof in the case where $P_1(h) = -h^2 \Delta_g,$ and $E_1=1$.

\subsubsection{Schr\"{o}dinger case} To treat the more general Schr\"{o}dinger case, we simply note that (see e.g. \cite{CHT}) in analogy with the homogeneous case in (\ref{sogge}), 
$$K_1(x,y) = (2 \pi h)^{(1-n)/2} e^{i r_E(x,y)/h}  \, { \hat{\rho}(r_E(x,y)) } \, a(x,y,h) + O_{C^\infty}(h^{\infty})$$
where $r_E(x,y)$ is Riemannian distance in the Jacobi metric $g_E = (E -V)_{+} g$ which is non-singular in the allowable region $ \{ V < E \};$ in particular, $r_E(x,y)$ locally satisfies the eikonal equation
$$ | d_z r_E(x,y)|_{g_E}^2 = 1; \quad  x \in \overline{\Omega},\,\, \epsilon < r_E(x,y) < 2 \epsilon,$$
with $\epsilon>0$ fixed sufficiently small. Consequently, using geodesic normal coordinates in $g_E$ centered at $x \in \overline{\Omega},$ it follows that the composite kernel $K(x,z)$ has exactly the same form as in (\ref{compositekernel}). The rest of the argument follows in the same way as in the homogeneous case.
 \end{proof}

 \subsection{Geometric implications of the Morse condition}
 
The morse assumption, Definition~\ref{morseassumption}, may at first seem artificial. However, we observe in section~\ref{examples} that it is satisfied in many examples and, moreover, it implies a purely geometric condition which is natural. In particular, for the QCI system $\hat{\mathcal{P}}$ and $x_0\in M$, there are $n$ natural submanifolds for $L^\infty$ norms:
$$
\Sigma^{E_i}_{x_0,i}:=p_i^{-1}(E_i)\cap T^*_{x_0}M,\qquad i=1,\dots n.
$$
Because we work with only two propagators, we consider $\Sigma^E_{x_0}=\Sigma^{E_1}_{x_0,1}\cap \Sigma^{E_2}_{x_0,2}$. The Morse condition does \emph{not} guarantee that $\Sigma_{x_0,1}\cap\Sigma_{x_0,2}$ is a transverse intersection (inside $T^*_xM$) indeed, not even that the intersection is clean. However, it does ensure that for \emph{every} energy $E_2$, the volume of $\Sigma^E_{x_0}$ small. More precisely (in dimension $n\neq 3$) it ensures that for every $E_2$, 
 $$
\Sigma_h:= \textup{Vol}\big(\{ \rho\in \Sigma^{E_1}_{x_0,1}\mid d(\rho, \Sigma^E_{x_0})<Ch\big)\leq C(h^{\frac{n-1}{2}} +h)
 $$
 Because $P_1u=E_1u$ and $P_2u=E_2u$, we can see that $u$ is localized in an $h$ neighborhood of $\{p_1=E_1,\,p_2=E_2\}$ and thus $\Sigma_h$ is the only region on which $u$ can have energy producing large $L^\infty$ norm at $x_0$ This volume localization then gives improved $L^\infty$ norms. 
 
The philosophy that volume concentration over $\Sigma_{x_0,1}^{E_1}$, implies improved $L^\infty$ norms can be made rigorous~\cite{CG18}. In future work~\cite{GT18}, we will use the ideas there to use directly the volume of the set $\Sigma_h$ to obtain a Hardy type bound for QCI eigenfunctions under a morse type assumption on the system.

%
%
%

\section{Exponential decay estimate for joint eigenfunctions in the microlocally forbidden region}
In this section, to prove our eigenfunction decay estimates, we will assume that $(M,g)$ is real-analytic and the QCI system $P_1(x, hD_x),...,P_n(x,hD_x)$ consists of analytic $h$-differential operators.  To formulate and prove our results, we will now recall some basic complex geometry and $h$-analytic microlocal machinery that will be used later on.

\subsection{Complex geometry} In this section, we require $M$ be a compact, closed, real-analytic manifold of dimension $n$  and $\widetilde{M}$ denote a Grauert tube complex thickening of $M$ with $M$ a totally real submanifold. By the Bruhat-Whitney theorem, $\widetilde{M}$ can be identified with $M^{\C}_{\tau} := \{ (\alpha_x, \alpha_\xi) \in T^*M; \sqrt{\rho}(\alpha_x,\alpha_\xi) \leq \tau \}$ where $\sqrt{2\rho} = |\alpha_{\xi}|_g$ is the exhaustion function $M^{\C}_{\tau}$, and we identify $\widetilde M$ with $M_{\tau}^{\C}$ using the complexified geodesic exponential map $ \kappa : M_{\tau}^{\C} \rightarrow \tilde{M}$ with $\kappa(\alpha) = \exp_{\alpha_x,\C}( i \alpha_{\xi})$
 Viewed on $\widetilde{M}$, the function $\sqrt{\rho}(\alpha) = \frac{-i}{2\sqrt{2}} r_{\C}(\alpha,\bar{\alpha}),$ which satisfies homogeneous Monge-Ampere and its level sets exhaust the complex thickening $\widetilde{M}$ (see \cite{GS1} for further details). 

\smallskip

\smallskip

 We consider a complexification of $T^*M$ of the form
 \begin{equation}
 \label{complex}
 \widetilde{T^*M}:= \{ \alpha; |\Im \alpha_x | < \tau, \,\, |\Im \alpha_{\xi}| \leq \frac{1}{C} \langle \alpha_\xi \rangle \}
 \end{equation}
 where $C \gg 1$ is a sufficiently large constant and  $T^*M \subset \widetilde{T^*M}$ is then a totally-real submanifold invariant under the involution $\alpha \mapsto \bar{\alpha}.$
 
 One has a natural complex symplectic form on $\widetilde{T^*M}$ given by
 $$\Omega^{\C} = d\alpha_x \wedge d\alpha_{\xi}, \quad (\alpha_x,\alpha_\xi) \in \widetilde{T^*M}.$$
 
 Given the complex symplectic form, $\Omega^{\C}$, there are some natural Lagrangian submanifolds of   $\widetilde{T^*M}$  that are of particular interest to us: First, there is the {\em $\C$-Lagrangian submanifold } 
 $$ \tilde{ \Lambda}: =  {\mathcal P}_{\C}^{-1}(E) , \quad E \in {\mathcal B}_{reg},$$
 where $  {\mathcal P}_{\C}  = (p_{1}^{\C},...,p_{n}^{\C})$ and $p_{j}^{\C}$ denotes the holomorphic continuation of $p_j$ to $\widetilde{T^*M}.$ When the context is clear, in the following we will sometimes simply write $p$ for the holomorphic continuation ${\mathcal P}_{\C}.$  The level set
 $${\mathcal P}^{-1}(E) \subset {\mathcal P}_{\C}^{-1}(E), \quad E \in {\mathcal B}_{reg}$$
 is an $\R$-Lagrangian submanifold and, as we have already pointed out, by the Liouville-Arnold theoerem, it is a finite union of $\R$-Lagrangian tori.
 
 We recall that a complex $n$-dimensional submanifold, $\Lambda_I,$ of $\widetilde{T^*M}$ is said to be {\em I-Lagrangian} if it is Lagrangian with respect to
$$ \Im \Omega^{\C} = \ical  \, d\alpha_x \wedge d \alpha_{\xi} = d\rcal \alpha_x \wedge d \ical \alpha_{\xi} + d \ical \alpha_x \wedge d \rcal \alpha_{\xi}, $$
where $\Omega^{\C} = d\alpha_{x} \wedge d \alpha_{\xi}$ is the complex symplectic form on $\widetilde{T^*M}$. We will denote the correponding complex canonical one form by
$$\omega^{\C} = \alpha_{\xi} d\alpha_x; \quad (\alpha_x,\alpha_{\xi}) \in \widetilde{T^*M}.$$

There are several examples of $I$-Lagrangians that will be of particular interest to us; these include, graphs over the real cotangent bundle $T^*M$ of the form
$$ \Lambda_I = \{  \alpha + i H_{G}(\alpha), \,\,\alpha  \in T^*M \}$$
where $H_G$ is the Hamilton vector field of a real-valued $G \in C^{\infty}_{0}(T^*M; \R).$

\subsection{Complex symplectic geometry near caustics of fold type}
\label{s:fold}
There is a natural $I$-isotropic associated with the integrable system ${\mathcal P} = (p_1,...,p_n)$ and the associated $\C$-Lagrangian $\tilde{\Lambda}.$ To define it we let $T^*M \otimes \C := \widetilde{T^*M}_M,$ the complexification of $T^*M$ in the {\em fibre} $\alpha_{\xi}$-variables only and set
 \begin{equation} \label{iso}
  \tilde{\Gamma}_I := \Lambda_{\C} \cap \Big( T^*M \otimes \C \, \Big).\end{equation}
  
  We will now consider the  case where $\pi: \Lambda_{\R} \to M$ has {\em fold} singularities. As we will show below, in such a case, one can describe the structure of $\tilde{\Gamma}_I$ in detail locally near the projection of the caustic set. 
\begin{defn} \label{caustic}
We define the {\em caustic set} to be the subset of the real Lagrangian $\Lambda_{\R}$ given by
$${\mathcal C}_{\Lambda}:= \{ \alpha \in \Lambda_{\R}; \, \rankr \,(d_{\alpha_{\xi}} p_1(\alpha),...,d_{\alpha_{\xi} } p_n(\alpha) ) < n \}.$$
In addition, we say that the caustic ${\mathcal C}_{\lambda}$ is of {\em fold} type if the projection $\pi_{\Lambda_{\R}}: \Lambda_{\R} \to M$ has fold singularities along $\mathcal{C}_{\Lambda}.$

\end{defn}

It follows from an implicit function theorem argument that, under the fold assumption on the caustic set, $\pi(\Lambda_\R) $ is a real $n$-dimensional stratified subset of $M$ with boundary, and moreover,

 $$ \partial \pi( \Lambda_{\R}) \subset \pi ({\mathcal C}_{\Lambda}).  $$\

To see this, we need only show that if $\alpha\in \Lambda_{\R}$ and $\rankr\,(d_{\alpha_{\xi}} p_1(\alpha),...,d_{\alpha_{\xi} } p_n(\alpha) ) =n$, then $\pi(\Lambda_{\R})$ contains a neighborhood of $\pi(\alpha)$. For this, observe that $H_{p_i}$, $i=1,\dots n$ are tangent to $\Lambda_{\R}$. In particular, the rank condition implies that $d\pi H_{p_i}$, $i=1,\dots n$ are linearly independent and hence $\pi:\Lambda_{\R}\to M$ is a local diffeomorphism.

\begin{rem} \label{generic}
In general, ${\mathcal C}_{\Lambda}$ is a stratified space. Under the fold assumption in (i),  one has a decomposition of the form
$ {\mathcal C}_{\Lambda} = \cup_{k=1}^N H_k,$\
where the $H_{k}$ are closed hypersurfaces (of real dimension $n-1$). We note that the fold assumption above is  generically satisfied in all of the QCI examples that we are aware of.
\end{rem}

 Under the fold type assumption on ${\mathcal C}_{\Lambda}$, one can locally characterize the structure of $\tilde{\Gamma}_I$ near the caustic set. To motivate the general result, it is useful to consider first the simple case of the harmonic oscillator.
 
\subsubsection{Harmonic oscillator} Consider the one-dimensional harmonic oscillator  with  $p^{\C}(x,\zeta) = \zeta^2 + x^2, \,\,\, (x,\zeta) \in \R \times \C$ \ and $E>0.$ 
In this case, letting $z \to \sqrt{z}$ denote the principal square root function with branch cut along the negative imaginary axis, we  have
$$ \tilde{\Gamma}_I = \Gamma_I \sqcup \Lambda_{\R},$$
where
$$ \Lambda_{\R} = \{ (x,\xi) \in \R \times \R; |x| \leq \sqrt{E}, \xi = \pm \sqrt{E - x^2} \},$$
which is a single ellipse, and
$$\Gamma_I= \{ (x,\zeta) \in \R \times \C; |x| > \sqrt{E},  \zeta = \pm i \, \sqrt{ x^2 - E} \}.$$
The latter set clearly has $4$ connected components. See Figure~\ref{f:hOsc} for a picture of these sets.\\

\begin{figure}
\vspace{-2cm}
\includegraphics[width=.75\textwidth]{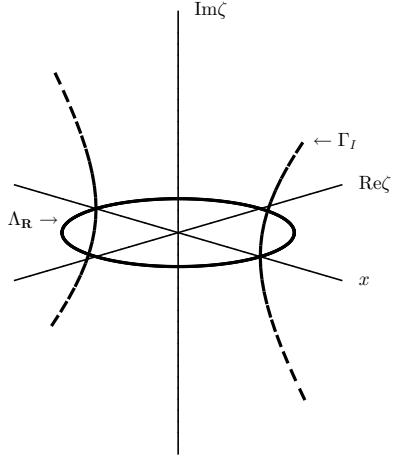}

\vspace{-2cm}
\caption{\label{f:hOsc} $\Lambda_{\R}$ and $\Gamma_I$ in the case of the harmonic oscillator}
\end{figure}

\begin{prop} \label{ilag} 
Assume that $(p_1,\dots p_n)$ are jointly elliptic and that the aperture constant $C$ in~\eqref{complex} is sufficiently large. Then
$ \tilde{ \Lambda} \cap  \Big( T^*M \otimes \C \Big)$ is compact and moreover, under the assumption that the caustic ${\mathcal C}_{\Lambda}$ is of fold type, there exists a neighbourhood $U$ of the caustic in $\tilde{\Gamma}_I$ such that
$$ (i) \,\,\, \tilde{\Gamma}_I \cap U =  ( \, \Lambda_{\R} \sqcup \Gamma_I \,) \cap U,$$
where $\Lambda_{\R} = \{ \alpha \in T^*M; {\mathcal P}(\alpha) = 0 \}$ and $\Gamma_I \subset \tilde{\Gamma}_I.$ Here,  both $\Lambda_{\R}$ and $\Gamma_I$ are $I$-isotropic submanifolds of the complex Lagrangian $\tilde{\Lambda}$ with respect to the complex symplectic form $\Omega^{\C}.$ 

In addition, $ \Gamma_I$ is locally a (complex) canonical graph with
$$ (ii) \,\,\,( \Gamma_I)_{U} = \{ (\alpha_x, d_{\alpha_x} \psi_{U}(\alpha_x) ); \,\,\, \alpha_x \in \pi(U)  \},$$
where $\psi_{U}: \pi(U) \to \C$ is a complex-valued, real-analytic function.
\end{prop}

\begin{rem} Here, $\Lambda_{\R}$ is, of course, also $\R$-Lagrangian with  respect to the {\em real} symptlectic form $\Omega$ on the real cotangent bundle $T^*M.$\end{rem}\

\begin{proof}

The fact that $\Lambda_{\C} \cap \Big( T^*M \otimes \C \Big)$ is compact follows readily from the joint ellipticity of the $p_j$'s. Indeed, since
$$ \Lambda_{\C} \cap (T^*M \otimes \C ) \subset \{ \alpha \in T^*M \otimes \C;  \sum_j |p_j(\alpha)|^2 = \sum_j E_j^2 \},$$
and by joint ellipticity, for all $\alpha \in T^*M,$
\begin{equation} \label{elliptic}
 \sum_j |p_j(\alpha)|^2 \geq \frac{1}{C'} |\alpha_{\xi}|^{2m},\end{equation}
it follows by Taylor expansion along $T^*M \subset T^*M \otimes \C$ and the fact that the $p_j$'s are symbols of $h$-differential operators (i.e. they are polynomials in the $\alpha_{\xi}$'s) that for $\alpha \in T^*M \otimes \C,$
\begin{equation} \label{taylor}
\sum_j |p_j(\alpha)|^2 = \sum_j |p_j(\alpha_x, \Re \alpha_{\xi})|^2 +  {\mathcal O}( |\Im \alpha_{\xi}| |\alpha_\xi|^{2m-1})).\end{equation}
Since $|\Im \alpha_{\xi}| \leq \frac{1}{C} |\Re \alpha_{\xi}|,$ and in view of (\ref{elliptic}), it follows that for aperture constant $C\gg 1$ sufficiently large, the second term on the RHS of (\ref{taylor}) can be absorbed in the first; the end result is that
\begin{equation} \label{ellipticcomplex}
 \sum_j |p_j(\alpha)|^2 \geq \frac{1}{C''} |\alpha_{\xi}|^{2m}, \quad \alpha \in T^*M \otimes \C
  \end{equation}
for some $m \in \Z^+.$ Thus, $ \tilde{\Lambda} \cap (T^*M \otimes \C)$ is clearly bounded since $M$ is compact and since it is also closed, compactness follows. 

To prove the remaining results (i) and (ii) in Proposition \ref{ilag}, we will use the fold assumption and argue in several steps.

Fix a point $q\in H_k\subset \mc{C}_{\Lambda}$. Then, by assumption $\pi_{\Lambda_{\R}}$ has a fold singularity and by \cite[Theorem C.4.2]{HOV3}, there are coordinates $y$ on $\Lambda_{\R}$ and $x$ on $M$ so that $y(q)=0$ and
\begin{equation}
\label{e:coord}
x(\pi(y))=(y_1,\dots y_{n-1},y_n^2).
\end{equation}
and in particular, locally, $H_k=\{y_n=0\}$. Now, since $\pi(x,\xi)=x$ for $(x,\xi)$ canonical coordinates on $T^*M$, we have that $x_i(y)=y_i$ for $i=1,\dots n-1$.  

Clearly, $\partial_{y_n}x_n|_{y=0}=0$ and, since $\Lambda_{\R}$ is Lagrangian,
  $$
 \sigma(\partial_{x_i}, \partial_{y_n})(q) = \sigma(\partial_{x_i},\sum_j\partial_{y_n}x_j(0)\partial_{x_j}+\partial_{y_n}\xi_j(0)\partial_{\xi_j})=0,\qquad i=1,\dots n-1.
$$ 
That is, $\partial_{y_n}\xi_i=0$, $i=1,\dots n-1$. Since $\partial_{y_1},\dots \partial_{y_n}$ are linearly independent, this implies that $\partial_{y_n}\xi_n|_{y=0}\neq 0$. 

Then, since the map $\kappa:(y_1,...,y_n) \mapsto (x'(y),\xi_n(y))$ satisfies rank $d\kappa = n,$,  by the implicit function theorem, $y_n=y_n(\xi_n,x')$ where $x=(x',x_n)$. Letting $b(x')=\xi_n|_{y_n=0}$, we can write using the implicit function theorem once again, 
$$y_n=\tilde{a}(x',\xi_n)(\xi_n-b(x'))$$
with $\tilde{a}(0)\neq 0$.

Therefore, we may choose coordinates $x$ on $M$ so that locally in canonical coordinates $(x,\xi)$,
\begin{equation} \label{normalform}
\pi_{\Lambda_{\R}}(x(x',\xi_n),\xi(x',\xi_n)) = (\,x',  a(x',\xi_n) \, ( \xi_n - b(x') )^2 \, );  \quad x = (x',x_n).
\end{equation}\
Here, $a \in C^{\omega}_{loc}(\R^n), \,\, a>0$ and $b \in C^{\omega}_{\loc}(\R^{n-1}).$

In this case, the caustic hypersurface is

$$ H_k = \{(x',\xi_n) \in \Lambda_\R; \quad \xi_n = b(x') \}.$$\

We note that under the projection $\pi_{\Lambda_{\R}},$ the hypersurface $H_k$ can naturally be identified with the hypersurface $\{ (x',x_n = 0) \in U \} \subset M$.  Henceforth, we abuse notation somewhat, and  denote the latter also by $H_k.$

Write 
$$a_2(x',\eta_n)=a(x',\eta_n+b(x')),$$ 
then the normal form (\ref{normalform}) can be rewritten in the form
\begin{equation} \label{normalform2}
\pi_{\Lambda_{\R}}(x(x',\xi_n),\xi(x',\xi_n)) = (\,x',  a_2(x',\xi_n - b(x')) \, ( \xi_n - b(x') )^2 \, );  \quad 0<  a_2 \in C^{\omega}_{loc}.
\end{equation}\

Next, we make a change of coordinates which will change the smooth structure near the caustic, but leave it unchanged away from the caustic. In particular, let $x_n=z^2$, $z\in \C$ so that 
$$z^2(x(x',\xi_n),\xi(x',\xi_n))=a_2(x',\xi_n - b(x')) \, ( \xi_n - b(x') )^2.$$
Note that when we want to return to the $x_n$ coordinates, we will write $\sqrt{x_n}=z$ where $\sqrt{x_n}>0$ for $x_n>0$ and the branch cut is taken on $-i[0,\infty)$. 
Then we have
$$
z=\pm\sqrt{a_2(x',\zeta_n-b(x'))}(\zeta_n-b(x')).
$$
and by the analytic implicit function theorem,   
\begin{equation}
\label{e:zeta}
\zeta_n^{\pm}=\zeta_n^{\pm}(x',z),\qquad z\in \mathbb{C}\text{ near }0.
\end{equation}
 Moreover, 
$$\pm\partial_{z}\zeta_n^{\pm}|_{z=0}=\frac{1}{\sqrt{a_2(x',0)}}>0.$$

 
  A simple computation using~\eqref{normalform2}, or more precisely its analytic continuation using $z$ as a coordinate, shows that
  $\pi_{\Lambda}: \Gamma_{I} \to M$ is locally surjective onto $M$ near the caustic hypersurface $H_k$. That is, there exists $W_k$ a neighborhood of $H_k$ in $\tilde{\Gamma}_I$ and $V_k$ a neighborhood of $\pi_{\Lambda}(H_k)$ so that 
  $$
  \pi_\Lambda:W_k\to V_k
  $$ 
  is surjective and, moreover, with $\Omega_k:=W_k\setminus H_k$, 
  \begin{equation} \label{fiber rank}
  \rankc \, ( \, d_{\zeta} p_1(x,\zeta),..., d_{\zeta} p_n(x,\zeta) \, ) = n, \quad (x,\zeta)  \in \Omega_k. \end{equation}
 To see this, we analytically continue~\eqref{e:coord}. In particular, analytically continuing $y\in \Lambda_{\R}$ to $\alpha\in\Lambda$, 
$$\alpha_x(\pi(\alpha))=(\alpha_1,\dots \alpha_{n-1},\alpha_n^2).$$
Hence, 
$$ \rankc d\pi_{\Lambda}=n,\qquad \alpha_n\neq 0.$$
Thus, $d\pi_\Lambda$ is surjective which implies that $\{d\pi H_{p_i}\}_{i=1}^n= d_\zeta p$ has rank $n$.

  
    We also note that $\pi |_{\Lambda}:\Omega_k \to M$ can be written as a graph over the base manifold $M$ locally near the caustic hypersurface $H_k$ up to choice of branch; more precisely, we have for some $\delta>0$
    
  \begin{eqnarray} \label{model2}
\Omega_k = \Omega_k^+ \cup  \Omega_k^{-}, \hspace{5cm}   &,  \\ \nonumber 
\Omega_k^{\pm} := \{ (x',z^2; \zeta' = \partial_{x'} \psi_U, \zeta_n =  \zeta_n^{\pm}(x',z) \, ) ; \quad z\in (0,\delta)\bigcup  i(0,\delta)  \}.
\end{eqnarray}
\begin{rem}
Note that $z^2\in\R $ for $z\in (0,\delta)\bigcup  i(0,\delta) $
\end{rem}

To complete the proof of Proposition \ref{ilag}, we will need the following result on solving a particular initial value problem for the complex eikonal equation associated with local branches $\Omega_k^{\pm}$ of the $I$-isotropic manifold $\Omega_k.$

\subsubsection{Complex generating functions}
In this section, we construct a  generating function  $\psi^{\pm}$ of $\Omega_k^{\pm}$ locally near the caustic hypersurface $H_k.$

Specifically, we seek to solve the complex eikonal boundary value problem 
\begin{align} 
p_j^{\C}(\alpha_x, \partial_{\alpha_x} \psi) &= E_j, \,\, j=1,...,n; \quad (\alpha_x,\partial_{\alpha_x} \psi) \in \Omega^{\pm}_k, \nonumber \\
 S |_{H_k} = 0; &\quad S = \Im \psi.  \label{eikonal} \end{align}
 
 \noindent In practice, we will not be able to find a unique solution $\psi$ on all of $\Omega_k$. However, for all such solutions, we will see that $S=\Im \psi$ agrees and hence that $S$ is well defined {on $\Omega_k.$}

\begin{lem} \label{eikonallemma}
Under the fold assumption on the real Lagrangian $\Lambda_\R$ (which is also $I$-isotropic), there exists 
{$S^{\pm} \in C^{1,1/2}_{loc}( \overline{\Omega_k^{\pm} }) \cap C^{\omega} ( \Omega_k^{\pm})$ so that $S=\Im \psi^{\pm}$ for any solution $\psi^{\pm}$ to} the complex eikonal  boundary value problem in (\ref{eikonal}). In addition, with $S^{\pm}=\Im \psi^{\pm}$,
\begin{equation}
\label{asymptotics}
S^{\pm}(x)=\pm\frac{2}{3\sqrt{a_2(x')}}(-x_n)_+^{3/2}+O(x_n^2).
\end{equation}
\end{lem}
\begin{proof}

To solve the eikonal problem, we follow the standard method of (complex) bicharacteristics. Since the caustic hypersurface $H_k$ is characteristic for the joint flow of Hamilton vector fields of $p_j^{\C}; j=1,...,n$, one cannot expect a smooth solution to (\ref{eikonal}). Nevertheless, it is still possible to solve (\ref{eikonal}), albeit with reduced regularity at $H_k.$
In normal coordinates $(x,\xi + i \eta)$, given an initial point $(x',\xi';0) \in H_k$ and $(x, \zeta) \in \Omega^{\pm}_k,$ we consider the ``normal'' curve joining these points given by
$$\gamma(t)=(x',tx_n;(\zeta')^{\pm}(x',\sqrt{tx_n}),\zeta^{\pm}_n(x',\sqrt{tx_n})),\quad t\in [0,1]$$
When $(x,\zeta) \in \Omega_k^{\pm},$ we write $\gamma^{\pm}$ for $\gamma$ to specify the branch. 
Let

\begin{equation}
\label{actionfn}
\begin{aligned}
\psi^{\pm}_k(x)&:=\int_{\gamma^{\pm}}\omega^{\C}=\int_{\gamma^{\pm}}\zeta dx=\int_0^1\zeta^{\pm}_n(x',\sqrt{tx_n})d(tx_n)=\int_0^{x_n}\zeta^{\pm}_n(x',\sqrt{x_n})dx_n
\end{aligned}
\end{equation}
Let 
$$S^{\pm}_k(x)=\int_{\gamma^{\pm}} \Im \omega^{\C} = \Im\int_0^{x_n}\zeta^{\pm}_n(x',\sqrt{x_n})dx_n $$
Now, $\pm \partial_{z}\zeta^{\pm}_n(x',s)|_{s=0}=\frac{1}{\sqrt{a_2(x')}}$, so 
$$
\zeta^{\pm}_n(x',z)=b(x')\pm\frac{z}{\sqrt{a_2(x')}}+O(z^2).
$$
In particular,
$$
S^{\pm}_k(x)=\pm\frac{2}{3\sqrt{a_2(x')}}(-x_n)_+^{3/2}+O(x_n^2).
$$

The fact that $\psi_{k}^{\pm}$ solves (\ref{eikonal}) on $\Omega_k^{\pm}$ respectively is clear from the definition above since from (\ref{model2}) $\Omega_{k}^\pm$ is locally a graph over $U_k$ with $\Omega_k^\pm = \{(x,\zeta); \zeta = \partial_{x} \psi_k^{\pm}(x) ) \}.$ 
 Here, of course, the function $\psi_U^{\pm} \equiv \psi_k^{\pm}$ for $x \in \Omega_k.$
Finally, from the formula in (\ref{actionfn}) it is clear that
$\psi_k^{\pm}, S_{k}^{\pm} \in C^{1,1/2}(\bar{\Omega_k}) \cap C^{\omega}(\Omega_k),$
since $H_k = \partial \Omega_k^{\pm} = \{ x \in U_k; x_n = 0 \}.$

We now show that the definition of  $S^{\pm}_k$ above is intrinsically defined in the sense that: (i) it is independent of choice of initial point on $H_k$ and (ii) it is independent of the choice of curve of integration in the same smooth homotopy class.

Indeed, to prove (i), we recall that $\zeta=\xi+i\eta$ and note that $\eta |_{H_k} = 0,$ so that if $\alpha_0, \alpha_{1} \in H_k$ and $\gamma(\alpha_0,\alpha_1) \subset H_k$ is a $C^1$-curve joining these points, then using that $H_k\subset T^*M$,
$$\int_{\gamma(\alpha_0,\alpha_1)} \Im \omega^{\C} = \int_{\gamma(\alpha_0,\alpha_1)} \eta \, dx = 0.$$

As for (ii), let $\gamma_1(\alpha_0,\alpha) \subset \Omega_k \cap \Omega_l$ and $\gamma_2(\alpha_0,\alpha) \in \Omega_k \cap \Omega_l$ be two homotopic smooth curves joining $\alpha_0 \in H_k$ to $\alpha \in \Omega_k \cap \Omega_l$. Then, since $\Omega_k \subset \Gamma_I$ is $I$-isotropic and $\Omega_k\cap \Omega_l\subset \widetilde{T^*M}_M$
it follows by Stokes formula that
$$ \int_{\gamma_1(\alpha_0,\alpha)} \eta \, dx = \int_{\gamma_2(\alpha_0,\alpha)} \eta \, dx.$$

\begin{rem}
Note that $\psi^{\pm}_k$ may depend on the choice of initial point in $H_k$, but  we have shown that $S_k^{\pm} = \Im\psi^{\pm}_k$ does not.
\end{rem}

\end{proof}

The fact that $\Gamma_I$ is  $I$-isotropic and  (i) and (ii)  clearly follow from Lemma \ref{eikonallemma} and that completes the proof of Proposition \ref{ilag}.

\end{proof}

\begin{defn} \label{d:actionfn}
From now on, we will refer to $S_k:= S_k^{+}$ as the {\em action function} corresponding to the caustic hypersurface $H_k.$ \end{defn}\

We extend $S_k$ to the entire caustic ${\mathcal C}_{\Lambda}$ be setting
$$ S_k (x) = 0, \quad x \in \pi({\mathcal C}_{\Lambda}),$$
so that, by definition, $S_k |_{H_{l}} = 0$ for all $l =1,...,N.$

\subsubsection{Action function corresponding to the entire caustic set ${\mathcal C}_\Lambda.$}

We now define the action function $S: \cup_k \pi(\Omega_k)  \to \R$ on the entire forbidden region $\cup_k \pi(\Omega_k).$ It remains to check that the $S_k$'s correponding to the different caustic hypersurfaces $H_k$ agree on overlaps. More precisely, we claim that 
\begin{equation} \label{overlaps}
S_k(\alpha) = S_l(\alpha), \quad \alpha \in \Omega_k \cap \Omega_l. \end{equation}\

The compatibility condition in (\ref{overlaps}) is readily checked: Let $\alpha_{0}^{k} \in H_k$ and $\alpha_{0}^{l} \in H_l$ and $\gamma(\alpha_0^k,\alpha_0^l) \subset H_k \cup H_l$ be a piecewise smooth curve inside the caustic joining $\alpha_0^k$ and $\alpha_0^l$  (which we recall is a {\em real} submanifold of $T^*M$).  Now let $\alpha \in \Omega_k \cap \Omega_l$ and $\gamma_1(\alpha_0^k,\alpha) \subset \Omega_k$ and $\gamma_2(\alpha_0^l,\alpha) \subset \Omega_l$ be two normal curves as above, Then, $ \gamma(\alpha_0^k,\alpha_0^l)  \cup \gamma_1(\alpha_0^k,\alpha) \cup \gamma_2(\alpha_0^l,\alpha)$ bounds a domain $\Omega_{kl} \subset \Omega_k \cap \Omega_l.$ Since $\Gamma_I$ is $I$-isotropic, it follows from Stokes formula that
\begin{equation} \label{stokes1}
 \int_{\gamma(\alpha_0^k,\alpha_0^l)} \eta \, dx + \int_{  \gamma_1(\alpha_0^k,\alpha)} \eta \, dx    -  \int_{\gamma_2(\alpha_0^l,\alpha)} \, \eta \, dx = 0. \end{equation}
However, since  ${\mathcal C}_{\Lambda} \subset T^*M$ so that $ \eta |_{ \gamma(\alpha_0^k,\alpha_0^l)  } =0,$ the first integral on the LHS of (\ref{stokes1}) vanishes and hence,
\begin{equation} \label{stokes2}
 \int_{ \gamma_1(\alpha_0^k,\alpha)} \eta \, dx = \int_{\gamma_2(\alpha_0^l,\alpha)} \, \eta \, dx.
 \end{equation}

 We now set

\begin{equation} \label{action}
S(\alpha_x):= \,S_k(\alpha_x); \quad \alpha_x \in \pi(\Omega_k). \end{equation}\

\noindent In view of the compatibility condition (\ref{overlaps}), the action function in (\ref{weight}) is well-defined. Also, from now on we denote the microlocally forbidden region by
$$ \Omega:= \cup_{k=1}^{N} \Omega_k.$$

\subsection{Analytic psdos and FBI transforms}
Let $U\subset T^*M$ be open. Following \cite{Sj}, we say that $a \in S^{m,k}_{cla}(U)$  provided $a \sim h^{-m} (a_0 + h a_1 + \dots)$ in the sense that
\begin{equation}
\label{scsymbol} 
\begin{gathered} 
\partial_{x}^{l_1} \partial_{\xi}^{l_2} \overline{\partial}_{(x,\xi)} a = O_{l_1, l_2}(1) e^{- \langle \xi \rangle/Ch}, \quad (x,\xi)\in U, \\ 
  \Big| \partial^\alpha \big (a - h^{-m} \sum_{0 \leq j \leq \langle \xi \rangle/C_0 h} h^{j} a_j \big)\Big| = O_{{\alpha}}(1) e^{- \langle \xi \rangle/C_1 h},\quad
 |a_j| \leq C_0 C^{j} \, j ! \, \langle \xi \rangle^{k-j},\qquad (x,\xi)\in U.
 \end{gathered}
 \end{equation}
We sometimes write $S^{m,k}_{cla}=S^{m,k}_{cla}(T^*M)$. 


We say that an operator $A(h)$ is a \emph{semiclassical analytic pseudodifferential operator of order $m,k$} if its kernel can be written as $A(x,y;h)=K_{1}(x,y;h)+R_{1}(x,y;h)$ where for all $\alpha,\beta$,
$$|\partial_x^\alpha \partial_y^\beta R_{1}(x,y,h)|\leq C_{\alpha\beta} e^{-c_{\alpha\beta}/h}, \,\,\, c_{\alpha \beta} >0,$$
and
$$K_{1}(x,y;h)=\frac{1}{(2\pi h)^{n}} \int e^{\frac{i}{h}\lan x-y,\xi \ran}a(x,\xi,h)\chi(|x-y|)d\xi$$
where $\chi\in \Cc(\re)$ is 1 near 0 and $a\in S^{m,k}_{cla}$. We say $A$ is $h$-elliptic if $|a_0(x,\xi)|>ch^{-m}\langle \xi\rangle^k$ where $a_0$ is from \eqref{scsymbol}. Recall also that $A$ is classically elliptic if there is $C>0$ so that if $|\xi|>C$, $|a_0(x,\xi)|>C^{-1}h^{-m}|\xi|^k$.
For more details on the calculus of analytic pseudodifferential operators, we refer the reader to \cite{SjA}.

 As in \cite{Sj},  given an $h$-elliptic, semiclassical analytic symbol $a \in S^{3n/4,n/4}_{cla}(M \times (0,h_0]),$  we consider an intrinsic FBI transform $T(h):C^{\infty}(M) \to C^{\infty}(T^*M)$ of the form
\begin{equation} \label{FBI}
T u(\alpha;h) = \int_{M} e^{i\phi(\alpha,y)/h}  a(\alpha,y,h)\chi( \alpha_x, y) u(y) \, dy \end{equation}
with $\alpha = (\alpha_x,\alpha_{\xi}) \in T^*M$  in the notation of \cite{Sj}.

\begin{rem}
The normalization $a\in S^{3n/4,n/4}_{cla}$ appears so that $T$ is $L^2$ bounded with uniform bounds as $h\to 0$~\cite{Sj}. 
\end{rem}

 The phase function is required to satisfy 
 \begin{equation}
 \label{e:Tform}
 \phi(\alpha,\alpha_x) = 0, \qquad \partial_y \phi(\alpha,\alpha_x) = - \alpha_{\xi},\qquad \Im (\partial_y^2 \phi)(\alpha,\alpha_x) \sim C |\langle \alpha_{\xi} \rangle| \, \Id.
 \end{equation}

Given $T(h) :C^{\infty}(M) \to C^{\infty}(T^*M)$ it follows by an analytic stationary phase argument \cite{Sj} that one can construct an operator $S(h): C^{\infty}(T^*M) \to C^{\infty}(M)$ of the form
\begin{equation} \label{left}
 S v(x;h) = \int_{T^*M} e^{-i  \, \overline{\phi(x,\alpha)}  /h} b(x,\alpha,h) v(\alpha) \, d\alpha \end{equation}
with $b \in S^{3n/4,n/4}_{cla}$ such $S(h)$  is a left-parametrix for $T(h)$ in the sense that
\begin{equation} \label{residual}
S(h) T(h) = \Id + R(h),\qquad\partial_{x}^{\alpha} \partial_{y}^{\beta} R(x,y,h) = O_{\alpha, \beta}(e^{-C/h}). \end{equation}

Henceforth, we  use  the invariantly-defined FBI transform
$T(h): C^{\infty}(M) \to C^{\infty}(T^*M)$ with phase function
\begin{equation} \label{FBIphase}
\phi(\alpha,y) =  \exp_{y}^{-1}(\alpha_x) \cdot \alpha_{\xi} +i \,\frac{\mu}{2} \, r^{2}(\alpha_{x},y) \langle \alpha_{\xi} /\mu\rangle. \end{equation}
Here, $\mu >0$ is a constant that will be chosen appropriately later, $r(\cdot,\cdot)$ is geodesic distance and 
$\chi(\alpha_x,y) = \chi_0(r(\alpha_x,y))$ where $\chi_0: \R \to [0,1]$ is an even cutoff with supp $\chi_0 \subset [-inj(M,g), inj(M,g)]$ and $\chi_0(r) =1$ when $|r| < \frac{1}{2} inj(M,g).$

In analogy with the above, when $\Lambda \subset \widetilde{T^*M}$ is an $I$-Lagrangian and with
$$ T_{\Lambda} u:= Tu |_{\Lambda},$$
one can also construct a left-parametrix 
$S_{\Lambda}(h): C^{\infty}(\Lambda) \to C^{\infty}(M)$ with the property that
\begin{equation} \label{inversion}
S_{\Lambda}(h) \cdot T_{\Lambda}(h) = Id + R_{\Lambda}(h) \end{equation}
where the Schwartz kernel of $R_{\Lambda}(h)$ satisfies the same exponential decay estimates as  $R(x,y,h)$ in (\ref{residual}).









%
%

\subsection{Weighted $L^2$-estimates along an $I$-Lagrangian}

First, given an analytic $h$-differential operator $P(x,hD) = \sum_{|\alpha | \leq k} a_{\alpha}(x) (h D_x)^{\alpha},$ an $I$-Lagrangian $\Lambda \subset \widetilde{T^*M}$ with generating function $H \in C^{\infty}(\Lambda;\R)$ satisfying
$$ dH  = \Im  \alpha_{\xi} d\alpha_x |_{\Lambda},$$ 
one has the following weighted $L^2$ estimate~\cite[Proposition 1.3]{Sj}

\begin{multline}
\label{ILAG}
 \langle e^{H/h} a  T_{\Lambda}(h) Q_1(h)   u_h,  e^{H/h} a T_{\Lambda}(h) Q_2(h)u_h \rangle_{L^2(\Lambda)} =  \langle q_1 |_{\Lambda} \, e^{H/h} a T_{\Lambda}(h) u_h, \, q_2 |_{\Lambda}\,e^{H/h}  a T_{\Lambda}(h) u_h \rangle_{L^2(\Lambda)}\\ \\
 +  O(h)  \| e^{H/h} T_{\Lambda}(h) u_h \|_{L^2(\Lambda)}^2, \quad a \in S^0(1). \\ \end{multline}

\noindent In (\ref{ILAG}), $q_i(\alpha) \in {\mathcal O}(\widetilde{T^*M})$ is the holomorphic continuation of the $h$-principal symbol of $Q_i(h)$ to $\widetilde{T^*M}$ and $q_i |_{\Lambda}$ is the restriction to the $I$-Lagrangian $\Lambda \subset \widetilde{T^*M}.$


For arbitrarily small but fixed $\epsilon>0$ and
$$\rho(x):= r(x,\pi(\Lambda_\R)),$$
 we let $\chi_{\epsilon} \in C^{\infty}(M;[0,1])$ be a cutoff with $\chi_{\epsilon}(x)=0$ when $r(x,\pi(\Lambda_\R)) \leq \epsilon/2$ and $\chi_{\epsilon}(x) = 1$ when $r(x,\pi(\Lambda_\R)) > \epsilon.$
 
 Let $\Omega$ be relatively open in $M$ with the property that $ \overline{\Omega}  \subset  M \setminus \pi(\Lambda_{\R})$ and $\overline{\Omega} \subset \{x; \rho(x) < \delta \}$ where $\delta >0$ will be subsequently  chosen sufficiently small independent of  $\epsilon >0.$ Let $\chi_{\Omega} \in C_{0}^{\infty}(M;[0,1])$ be a cutoff function with the property that
 $\chi_{\Omega}(x) = 1$ for $x \in \pi(\Lambda_{\R}) \cup \tilde{\Omega}$ and $\chi_{\Omega}(x) = 0$ for $x \in ( \pi(\Lambda_{\R}) \cup \Omega )^{c}$ where $\tilde{\Omega} \Subset \Omega$ is a small neighbourhood of projection $\pi(\Lambda_\R) \subset M.$
 

We assume here that the real Lagrangian $\Lambda_\R$ has a caustic set of fold type and then consider the particular {\em weight function} $H_{\epsilon} \in C^{\infty}(M;\R)$ given by

\begin{equation} \label{weight}
H_{\epsilon}(\alpha_x):= \, (1- \epsilon) \, S(\alpha_x) \cdot \chi_{\epsilon}(\alpha_x), \quad \alpha_x \in \Omega, \end{equation}\\
where $\psi^+: \Omega \to \C$ solves the complex eikonal equation in (\ref{eikonal}) and the branch is chosen so that $\Im \psi^+ = S.$ The associated  $I$-Lagrangian is

\begin{equation}
\label{e:lambdaE}
\Lambda_{\e}:= \{ (\alpha_x,  \, \alpha_{\xi} + i\partial_{\alpha_x} H_{\epsilon}(\alpha_x) ); \,\,\, \alpha \in T^*M \}.
\end{equation}\

Let $u_h \in C^{\infty}(M)$ be a joint eigenfunction (or exponential quasimode) of $P_j(h); j=1,...,n$ with $P_j(h) u_h = O(e^{-C/h})$ (nb: we have normalized the operators $P_j(h)$ here so that the joint eigenfunctions $u_h$ have joint eigenvalues all zero). An application of the weighted estimate (\ref{ILAG})  applied with $a = \chi_{\Omega}$, $Q_1=Q_2= P_j(h)$ and then summed over $j=1,..,n$ gives

\begin{multline}
\label{ILAG2}
 \langle q \,\, \chi_{\Omega} e^{H_{\epsilon}/h} T_{\Lambda_\e}(h) u_h, \,  \,\, \chi_{\Omega} e^{H_{\epsilon}/h}  T_{\Lambda_\e}(h) u_h \rangle_{L^2(\Lambda_\e)}\\ +  O(h)  \| \chi_{\Omega} e^{H_{\epsilon}/h} T_{\Lambda_\e}(h) u_h \|_{L^2(\Lambda_\e)}^2 = O(e^{-C/h}),  \end{multline}
where
\begin{equation} \label{taylor1}
q(\alpha) = \sum_{j=1}^n \big |p_j |_{\Lambda_\e}\big|^2 (\alpha) = \sum_{j=1}^n {|p_j (\alpha_x, \alpha_{\xi} + i \partial_{\alpha_x} H_{\epsilon}(\alpha_x) ) |^2. } \end{equation}\

Splitting the LHS of (\ref{ILAG2}) into pieces where $\rho > \epsilon$ and $\rho < \epsilon$ and noting that $ \Im H_{\epsilon}(\alpha_x)< c\e^{3/2}$ when $\rho(\alpha_x) < \epsilon$  and $\Im H_{\epsilon}(\alpha_x) = (1-\epsilon) \cdot S(\alpha_x)$ when $\rho(\alpha_x) > \epsilon$ gives with appropriate  $\beta(\epsilon) = O(\epsilon^{3/2}),$ 

\begin{equation}
\label{ILAG3}
\begin{aligned}
&\big\langle q \, {\bf 1}_{\rho > \epsilon}  \,\,\chi_{\Omega} e^{H_\e/h} T_{\Lambda_\e}(h) u_h, \,  \chi_{\Omega} e^{H_{\epsilon}/h}  T_{\Lambda_\e}(h) u_h \big\rangle_{L^2(\Lambda_\e)} + O(h)  \| e^{H_\e/h}  {\bf 1}_{\rho > \epsilon} \chi_{\Omega} T_{\Lambda_\e}(h) u_h \|_{L^2(\Lambda_\e)}^2  \\
  &\qquad\qquad\qquad\qquad\qquad\qquad\qquad= O(e^{\beta(\epsilon)/h}) \|  {\bf 1}_{\rho  \leq \epsilon}  \chi_{\Omega} T_{\Lambda_\e} u_h \|_{L^2(\Lambda_\e)}^2 + O(e^{-C/h}) \\ 
  &\qquad\qquad\qquad\qquad\qquad\qquad\qquad= O(e^{\beta(\epsilon)/h}) \| \chi_{\Omega} T_{\Lambda_\e} u_h \|_{L^2(\Lambda_\e)}^2 + O(e^{-C/h}). 
   \end{aligned}
\end{equation}

In the last line of (\ref{ILAG3}), we have used some elementary bounds on $S$; indeed,  from (\ref{asymptotics}) that as $\rho \to 0^+,$
\begin{equation*} 
S(x) = O(\rho(x)^{3/2}),  \end{equation*}\\
as $\rho \to 0^+,$ where  $\rho(\alpha_x) = d_g(\pi(\Lambda_{\R}), \alpha_x).$ We will also need
\begin{equation}\label{genbound}
\quad  \partial_x S(x) = O(\rho(x)^{1/2})
\end{equation}
From (\ref{genbound}) and the formula for $\Lambda_\e$ and $T_{\Lambda_\e}$~\eqref{e:lambdaE} and~\eqref{e:Tform} respectively, together with the fact that $T_{_{T^*M}}:L^2\to L^2$ is uniformly bounded in $h$, it follows that
$$ \|  {\bf 1}_{\rho \leq \epsilon} \chi_{\Omega} T_{\Lambda_\e} u_h \|_{L^2(\Lambda_\e)}^2  \leq C \sup_{\rho \leq \epsilon} e^{ 2 |\partial S(\rho)| /h}.$$

Thus, in view of (\ref{genbound}), the RHS of (\ref{ILAG3}) is $O(e^{\beta'(\epsilon)/h})$ where $\beta'(\epsilon) = O(\epsilon^{1/2})$ as $\epsilon \to 0^+$ and so, it follows from (\ref{ILAG3}) that

\begin{multline}
\label{ILAG3.5}
\big\langle q \, {\bf 1}_{\rho > \epsilon}  \,\,\chi_{\Omega} e^{H_\e/h} T_{\Lambda_\e}(h) u_h, \,  \chi_{\Omega} e^{H_{\epsilon}/h}  T_{\Lambda_\e}(h) u_h \big\rangle_{L^2(\Lambda_\e)} + O(h)  \| e^{H_\e/h}  {\bf 1}_{\rho > \epsilon} \chi_{\Omega} T_{\Lambda}(h) u_h \|_{L^2(\Lambda)}^2  \\
= O(e^{\beta'(\epsilon)/h}),
 \end{multline}
where $\beta'(\epsilon) = O(\epsilon^{1/2})$ as $\epsilon \to 0^+.$



 

We will need the following

\begin{lem} \label{elliptic}

Let $\Omega \subset M \setminus \pi(\Lambda)$ with  $\Omega \subset \{x: \e < \rho(x) <  \delta \}$. Then, under the fold assumption on ${\mathcal C}_{\Lambda},$ there exists a {\em fixed}  $\delta_0>0$ so that for $0<\e<\delta<\delta_0$ there exists $c>0$ so that 
$$ |q(\alpha)| \geq c \langle\alpha_\xi\rangle^{2m}>0, \quad \text{when} \,\, \alpha_x \in \Omega,$$
\end{lem}

\begin{proof}

We assume throughout that $\epsilon < \rho(\alpha_x) < \delta,$ so that, in particular the weight function $H(\alpha_x) = ( 1- \epsilon)  \,  {S}(\alpha_x)$.  {Since we may work locally, we let $\psi^+$ be a solution to~\eqref{eikonal} near $\alpha_x$ so that in particular, $\psi^+=\Re \psi^++iS.$ }

Case (i) $|\alpha_\xi -\Re \partial_{\alpha_x}\psi^+ |\ll1$: First, observe that in a neighborhood of the caustic $\mc{C}_\Lambda$, the \emph{only} solutions to 
$p_j(x,\zeta)=0$, $j=1,\dots n$ occur at $\zeta=\zeta^{\pm}(x',\sqrt{x_n})$ where
$$\zeta^{\pm}=(\zeta'(x',\zeta_n^{\pm}(x',\sqrt{x_n})),\zeta_n^{\pm}(x',\sqrt{x_n}))$$
and $\zeta_n^{\pm}$ is as in~\eqref{e:zeta}. Therefore, there is $\delta_0>0$ and $c=c(\delta_0)>0$ so that with
$$\Lambda_{\epsilon}(c(\delta_0)):=  \{ (\alpha_x,  \, \alpha_{\xi} +i \partial_{\alpha_x} H_{\epsilon}(\alpha_x) ); \,\,\,|\alpha_{\xi}- \Re \partial_{\alpha_x}\psi^+ | \leq c(\delta_0), \,\, \alpha_{\xi} \in T_{\alpha_x}^*M \},$$
and $\alpha\in \Lambda_{\epsilon}(c(\delta_0))$ with $\e<\rho(\alpha)<\delta<\delta_0$,
$$
|q(\alpha)|>c_{\e,\delta}>0.
$$

Case (ii) $|\alpha_\xi -\Re \partial_{\alpha_x}\psi^+ |\gg 1$: Since $p_j$, $j=1,\dots n$ are symbols of order   $m$, $\partial_\xi |p_j|^2 (x,\xi)|\leq C\langle \xi\rangle^{2m-1}$.  Moreover, $q=\sum_j p_j^2$ is classically elliptic. Therefore, $|q(x,\xi)|\geq c\langle \xi\rangle^{2m}-C$. Now, 
\begin{align*}
q(\alpha)&=\sum {|p_j}(\alpha_x, \alpha_\xi-\Re\partial_{\alpha_x}\psi^++\e\Re\partial_{\alpha_x}\psi^++(1-\e)\partial_{\alpha_x}\psi^+)|^2\\
&=\sum {|p_j}(\alpha_x,\alpha_\xi -\Re\partial_{\alpha_x}\psi^++\e\Re \partial_{\alpha_x}\psi^++(1-\e)(\Re \psi^++i\partial_{\alpha_x}S)|^2\\
&=\sum {|p_j}(\alpha_x,\alpha_\xi -\Re \partial_{\alpha_x}\psi^+)|^2\\
&\qquad +O(|\alpha_\xi|^{2m-1}(\|\partial_{\alpha_x} S\|_{L^\infty(\e<\rho<\delta)} + \|\partial_{\alpha_x}\Re\psi^+\|_{L^\infty(\e<\rho<\delta)})\\
&\geq c|\alpha_\xi-\Re\partial_{\alpha_x}\psi^+|^{2m}-C_\delta
\end{align*}
since $\|\partial_{\alpha_x} S\|_{L^\infty(\e<\rho<\delta)}+ \|\partial_{\alpha_x}\Re\psi^+\|_{L^\infty(\e<\rho<\delta)} <C_\delta$ In particular, there exists $C=C(\delta_0)>0$ so that if $|\alpha_\xi -\Re\partial_{\alpha_x}\psi^+|>C(\delta_0)$ and $\e<\rho(\alpha)<\delta<\delta_0$, then $|q|>c_{\delta_0}|\alpha_\xi|^{2m}$.

Case(iii): Assume $ c(\delta_0) \leq |\alpha_{\xi}-\Re \partial_{\alpha_x}\psi^+| \leq C(\delta_0).$ In this case, we let
$$ \Lambda_{\epsilon}(c(\delta_0), C(\delta_0)):= \{ (\alpha_x,  \, \alpha_{\xi} + i\partial_{\alpha_x} H_{\epsilon}(\alpha_x) ); \,\,\, c(\delta_0) \leq |\alpha_{\xi} -\Re \partial_{\alpha_x}\psi^+| \leq C(\delta_0), \,\, \alpha_{\xi} \in T_{\alpha_x}^*M \}.$$

To control $|q(\alpha)|$ on this set, let 
$$\tilde{\Lambda}(c(\delta_0),C(\delta_0))=\{(\alpha_x\,\alpha_\xi)\mid c(\delta_0)\leq |\alpha_\xi-\Re \partial_{\alpha_x}\psi^+|\leq C(\delta_0)\}$$
Note that since $\Omega\cap \pi(\Lambda_{\R})=\emptyset$, and $\tilde{\Lambda}(c(\delta_0),C(\delta_0))\cap \overline{\Omega}$  is compact, 
$$
\inf_{\alpha^0\in\tilde{\Lambda}(c(\delta_0),C(\delta_0))\cap \Omega} \sum |p_j(\alpha^0_x,\alpha^0_\xi)|^2>0.
$$
Then, for $\alpha\in \Omega\cap \Lambda_{\epsilon}(c(\delta_0),C(\delta_0)),$ there is $\alpha_0\in \Omega\cap \tilde{\Lambda}(c(\delta_0),C(\delta_0))$ so that  
$$q(\alpha)=\sum |p_j(\alpha^0_x,\alpha^0_\xi)|^2+O(\delta^{1/2}).$$
In particular, there is $\delta_1>0$ so that for all $0<\delta<\delta_1$, and $\alpha\in \Omega\cap \Lambda_{\epsilon}(c(\delta_0),C(\delta_0)),$
$$
|q(\alpha)|>c>0.
$$

%
%
%
%
%
%
\end{proof}


\subsection{Proof of Theorem \ref{expdecay}}
\begin{proof}
Without loss of generality, we assume here that supp $\, \chi_{\Omega} \subset \{ \rho < \delta \}.$ Then,
In view of Lemma~\ref{elliptic}, it follows from (\ref{ILAG3.5}) together with that fact that on $\supp \chi_\Omega$, $|(1-\e)S-H_\e|=O(\e^{3/2})$, that for $\epsilon >0$ sufficiently small and $h \in (0,h_0(\epsilon)],$

\begin{equation} \label{ILAG4}
\| e^{(1-\e)S/h} {\bf 1}_{\epsilon < \rho < \delta } \,  \chi_{\Omega} T_{\Lambda} u_h \|_{L^2(\Lambda)} = O(e^{\beta'(\epsilon)/h}) + O(e^{-C/h}),\end{equation}
where $\beta'(\epsilon) = O(\epsilon^{1/2})$ as $\epsilon \to 0^+.$

Thus, it follows that 

\begin{equation} \label{ILAGupshot}
\| e^{(1-\e)S/h}  \chi_{\Omega} T_{\Lambda} u_h \|_{L^2(\Lambda)} = O_{\epsilon} (e^{\beta(\epsilon)/h}), \quad \beta(\epsilon) = O(\epsilon^{1/2}).\end{equation}

\begin{rem} The argument as above works in semiclassical Sobolev norm in the same way, with
\begin{equation} \label{ILAG5}
\| e^{(1-\e)S/h}  \,  \chi_{\Omega} T_{\Lambda} u_h \|_{H^{m}_{h}(\Lambda)} = O_{m,\epsilon}(e^{\beta(\epsilon)/h}). \end{equation}
\end{rem}

\noindent In both (\ref{ILAG4}) and (\ref{ILAG5}) $\beta(\epsilon) = O(\epsilon^{1/2})$ as $\epsilon \to 0^+.$




Let $\psi\in C_c^\infty(\re^n)$ so that $|q|\geq c\langle \alpha_\xi\rangle^m$ on $\supp (1-\psi)(\alpha_\xi)$. Such a $\psi$ exists by Lemma~\ref{elliptic}.  Standard elliptic estimates for analytic pseudos (see e.g.~\cite[Proposition 2.2, Corollary 1.3]{GalkToth},~\cite[Theorem 4.22]{MarBook}) together with the fact that $P_iu=0$  shows that there exists $h_0(\mu)$ such that for $h \in (0,h_0(\mu) )$  such that
\begin{equation} \label{massestimate}
\|\chi_\Omega (1-\psi(\alpha_\xi)) T_{\Lambda} u\|_{L^2(T^*M)}=O(e^{-C/h}).
\end{equation}
Moreover, as we show in the appendix, the exponential rate constant $C>0$ can be chosen {\em uniformly} for all $\mu \geq \mu_0 >0,$ $h<h_0(\mu)$ where $\mu$ is the constant appears in the phase function in (\ref{FBIphase}) (see Proposition~\ref{p:unifElliptic}).

In particular, since $(|S|+|H_\e|+|\partial_{\alpha_x}H_\e|) \leq C\delta^{1/2}$, this implies that there is $\delta>0$ and $\mu_0>0$ so that for all $\mu>\mu_0$, 
\begin{equation}
\label{e:outside}
\|e^{(1-\e)S/h}\chi_\Omega(1-\psi(\alpha_\xi))T_{\Lambda}u\|\leq e^{-C/h}, \,\,\, C>0.
\end{equation}
We also note that 
$$
\|S_{\Lambda}\chi_{\Omega}\|_{L^2(\Lambda)\to L^2(M)}\leq Ce^{\sup_{\Omega} |\Im \partial_\alpha S|}\leq Ce^{ \delta^{1/2}/h}.
$$

Let $\chi_{1,\Omega}$ supported on $\chi_\Omega\equiv 1$ and $\chi_{2,\Omega}\equiv 1$ on $\supp \chi_{\Omega}$ with $\chi_{i,\Omega}\in C_c^\infty(\Omega)$. Then, as we show in the Appendix, there is $\delta>0$ so that for $\mu>\mu_0$, one can construct a left-parametrix $S_{\Lambda}: C^{\infty}_{0}(T^*M) \to C^{\infty}(M)$ with the property that for some uniform constant $C>0,$ 
\begin{equation}
\begin{aligned}
\label{inversion}
e^{ (1-\e)S/h} \chi_{1,\Omega} u_h &= e^{ (1-\e)S/h} \chi_{1,\Omega} S_{\Lambda} T_{\Lambda}u_h + O(e^{-1/Ch})   \\
&= e^{ (1-\e)S/h} \chi_{1,\Omega} S_{\Lambda} \chi_\Omega T_{\Lambda }u_h + O(e^{-1/Ch}) \\
&= e^{ (1-\e)S/h} \chi_{1,\Omega} S_{\Lambda} \psi(\alpha_\xi) \chi_\Omega T_{\Lambda}u_h \\
&\qquad\qquad\qquad +e^{ (1-\e)S/h} \chi_{1,\Omega} S_{\Lambda}(1- \psi(\alpha_\xi) )\chi_\Omega T_{\Lambda}u_h + O(e^{-1/Ch })   \\
&= \big( e^{ (1-\e)S/h}\chi_{1,\Omega} \,S_{\Lambda} \, e^{- (1-\e)S/h} \psi(\alpha_\xi)\chi_{2,\Omega}(\alpha_x)\big) \cdot \big( e^{  (1-\e)S/h}  \chi_\Omega T_{\Lambda} \big) u_h + O(e^{-1/Ch }). 
\end{aligned}
\end{equation}

Here, we recall the exponential constant $C>0$ in the remainder terms in (\ref{inversion}) does not depend on the constant $\mu>0$ in the phase function (\ref{FBIphase}) of the FBI transform which we now fix large enough, with
\begin{equation} \label{setmu}
\frac{\mu}{2} >  \| \partial^2 S \|_{L^\infty(\Omega)}:= \max_{x \in \Omega} | \partial_{x_i} \partial_{x_j} S(x)|. \end{equation}


Consequently from (\ref{ILAG4}), the  Cauchy Schwarz inequality and the last line of (\ref{inversion}) one gets that for $x \in \Omega,$ and any $\epsilon >0,$


\begin{equation} \label{upshot}
  | e^{ (1-\e)S/h}\chi_{1,\Omega} u_h(x) | \leq C_{\epsilon} e^{\beta(\epsilon)/h} \sup \| A_{\Lambda}(x,\cdot ;h)\|_{L^2(\Lambda)}  + O(e^{-C_1/h}), \quad \beta(\epsilon) = O(\epsilon^{1/2}). \end{equation}

Here, $A_{\Lambda}(x,\alpha;h)$ is the Schwartz kernel of the operator $A_{\Lambda}(h): C^{\infty}(\Lambda) \to C^{\infty}(M)$ where

\begin{equation} \label{composite}
A_{\Lambda}(h):=  e^{ (1-\e)S/h} \chi_{1,\Omega}\cdot S_{\Lambda}(h) \cdot e^{- (1-\e)S/h}\psi(\alpha_\xi)\chi_{2,\Omega}(\alpha_x). \end{equation}

Consequently, it remains to bound   $ \|   A_{\Lambda}(h) \|_{ L^2(\Lambda) \to L^\infty(M)}.$ We note that by Lemma~\ref{eikonallemma} under the fold assumption, we can find local coordinates $x = (x',x_n): \Omega \to \R^n$ in a neighbourhood, $\Omega$ of the caustic in terms of which
$$ S(x) = b(x',x_n) x_n^{3/2}; \quad 0< b \in C^{\omega}(\Omega).$$
 
 
By Taylor expansion,

$$ S(x) - S(\alpha_x) - \langle \partial S(\alpha_x), x- \alpha_x \rangle  \leq \| \partial^2 S \|_{\infty} |x-\alpha_x|^2,$$

It follows that for $x \in \Omega,$ and with appropriate $m>0,$ 

\begin{equation}
\begin{aligned}
\label{kernelbound}
& \int_{\Lambda} | A_{\Lambda}(x(y),\alpha;h) |^2 \,d\alpha \,   \\ \\
 &\qquad \leq C h^{-m}    \int_{T^*M} \Big|  e^{-2 i \phi^*(\alpha,y)/h} \, e^{[2 (1-\epsilon)S(x) - 2 (1-\epsilon)S(\alpha_{x} )    - 2(1-\epsilon) \langle \partial_{\alpha_x} S(\alpha_x), \, x - \alpha_x \rangle \,  ]/h} \, \Big| \,\\
 &\qquad\qquad\qquad\qquad\qquad \times \chi(r(\alpha_x,x)) \chi_{1,\Omega}(x)\chi_{2,\Omega}(\alpha_x)\,  \psi(\alpha_\xi) \, {\bf 1}_{\rho \geq \epsilon} (\alpha_x) d\alpha  \\ \\
&\qquad\leq C h^{-m}  \int_{T^*M}   e^{ \big( 2 \ical \phi^*(\alpha,y)  + \| \partial^2 S \|_{\infty} |x-\alpha_x|^2 \big) /h}  \, \chi(r(\alpha_x,x)) \, \chi_{1,\Omega}(x)\,\chi_{2,\Omega}(\alpha_x)\,  \psi(\alpha_\xi) {\bf 1}_{\rho \geq \epsilon} (\alpha_x) d\alpha \\ \\
 &\qquad\leq C h^{-m} \int_{T^*M}  e^{ \big( \, -\frac{\mu}{2} + \| \partial^2 S \|_{\infty} \, \big) \, |x-\alpha_x|^2 /h} \, \chi(r(\alpha_x,x))   \chi_{1,\Omega}(x)\,\chi_{2,\Omega}(\alpha_x)\,  \psi(\alpha_\xi) \, d\alpha = O(h^{-m + \frac{n}{2}})
  \end{aligned}
  \end{equation}

\noindent uniformly for $x \in \supp \chi_{1,\Omega}.$  The last line follows by an application of steepest descent under  the assumption (\ref{setmu}) on the constant $\mu >0$ in the phase function $\phi(\alpha,x).$    
   
    Thus, in particular, it follows that for any  $\Omega  \subset M \setminus \pi(\Lambda_{\R})$ sufficiently close to the caustic $\partial  \pi (\Lambda_{\R}),$

  
  




\begin{equation} \label{opbound}
 \|  A_{\Lambda}(h) \|_{L^2(\Lambda) \to L^{\infty}(M)} = O(h^{-m'}) \end{equation}
with some $m' >0.$ Thus, in view of (\ref{opbound}) and (\ref{upshot}), we have proved Theorem \ref{expdecay}. \end{proof}



 

\begin{rem} \label{nonfold} Many classical integrable systems (eg. geodesic flow on ellipsoids, Neumann oscillators on spheres, geodesic flow on Liouville tori),  have the feature that in terms of appropriate coordinates $x=(x_1,...,x_n) \in \prod_{j=1}^n (\alpha_j,\alpha_{j+1})$ with $\alpha_1 < \alpha_2< \cdots \alpha_n$ defined in a neighbourhood, $V,$ of $\pi(\Lambda_{\R})$ one can separate variables in the generating function $S_{V}: V \to \R$ with
$$p_j(x, d_x S_V(x)) = E_j, \quad S_V(x) = \sum_{j=1}^n S_V(x_j), \,\, x \in V.$$
Moreover, one can write each $S_V(x_j)$ as a hyperelliptic integral
$$S_V(x_j) = \int_{\alpha_j}^{x_j} \sqrt{ \frac{R_E(s)}{A(s)} } \, ds,$$ 
where $R_E$ is a polynomial of degree $n-1$ with  with coefficients that depend on the joint energy levels $E=(E_1,...,E_n) \in {\mathcal B}_{reg}$ When $n=2$ the roots of $R_E(s)$ are necessarily simple (since it is linear) and this is generically still  the case in higher dimensions as well.

 The proof of Theorem \ref{expdecay} holds in the (non-generic) case where $R_E(s)$ has multiple roots. Indeed, in the case where $R_E(s)$ has a root $r_k \in (\alpha_k, \alpha_{k+1})$ of mulitiplicity $2k+1$ corresponds to  a caustic hypersurface $H_k  = \{ x_k = r_k \}$ with $\Omega_k = \{ x_k  > r_k \}.$ The complex generating function near $H_k$ in the analogue of Lemma \ref{eikonallemma}
is then  locally of the form 
$$S(x) \sim a(x',x_k) (x_k - r_k)^{k + 3/2}; \quad a (x)>0, \,\, x \in \Omega_k.$$
Consequently, both $S |_{x_k = r_k} = 0$ and $d S |_{x_k = r_k} =0$ and
also $d S(x_k) \neq 0$ when $x_k > r_k,$ the reader can readily check that the analogue of Lemma \ref{elliptic} holds in this case also and the proof of Theorem \ref{expdecay} then follows in the same way as in the fold case where $k=0.$

 \end{rem}

\section{Examples}\label{examples}

We begin with some relatively simple examples of QCI systems in two dimensions:  Laplace eigenfunctions on convex surfaces of revolution and Liouville tori/spheres. In these special examples, one can justify separation of variables for the joint eigenfunction that allow us to verify the sharpness of both Theorems \ref{QCI} and \ref{expdecay}.

\subsection{Convex surfaces of revolution} Consider a convex surface of revolution generated by rotating a curve  $ \gamma = \{ (r, f(r)), \, r \in [-1,1] \}$ about $r$-axis with $f\in C^\infty ([-1,1], \R)$, $f(1)=f(-1)=0$, ${f^{(2k)}(1)=f^{(2k)}(-1)=0}$, where $k$ is a nonnegative integer and $f''(r)<0$ for all $r\in (-1,1)$. Moreover, we will assume that $f(r)$ has a single isolated critical point at $r=0$; in particular, $f'(0)=0$ and $f''(0)<0.$
 
 Let $M$ be the corresponding convex surface of revolution parametrized by 
 \begin{align*}
 &\beta: [-1,1]\times [0, 2\pi) \to \R^3, \\
 &\beta(r,\theta) = (r, f(r) \cos \theta, f(r) \sin \theta).
 \end{align*}
 
Consider $M$ endowed with the rotational Riemannian metric $g$ given by
 \[g=dr^2+f^2(r)d\theta^2,\]
where $w(r)=\sqrt{1+(f'(r))^2}$.

The corresponding $h$- Laplacian $P_1(h) := - h^2 \Delta_g $ with eigenvalue $E_1(h) =1$  is QCI with commuting quantum integral $P_2(h) = h D_{\theta}$ and since the eigenfunctions can be expanded in Fourier series in $\theta,$ the joint eigenfunctions  are necessarily of the form
$\phi_h(r,\theta)=v_h(r)\psi_h(\theta),$ where   $v_h(r)$ and $\psi_h(\theta)$ must satisfy the ODE
\begin{equation} \label{tang}
 h D_{\theta} \psi_h(\theta)= E_2(h) \psi_h(\theta); \quad E_2(h) = m h,
\end{equation}
and
\begin{equation} \label{reduced}
\big( \, h^2 D_r^2  +  f^{-2}(r)E_2^2(h) - 1\, \big)v_{h}(r)=0.
\end{equation}

At the classical level, $p_1(r,\theta;\xi_r, \xi_{\theta}) =  \xi_r^2 - f^{-2}(r) \xi_{\theta}^2 $ and $p_2(r,\theta;\xi_r,\xi_{\theta}) = \xi_{\theta}$ with
$$ \Lambda_{\R}(E) = \{ (r,\theta; \xi_r,\xi_{\theta});  \xi_r^2  = 1 - f^{-2}(r)  \xi_{\theta}^2, \quad \xi_{\theta} = E_2 \}.$$

\subsubsection{Sup bounds} Set $\Sigma_{r,\theta}:= \{ (\xi_r,\xi_{\theta}); \in T_{r,\theta}^*M; p_1(r,\theta;\xi_r,\xi_{\theta}) = 1 \}. $  It is then clear that $p_2 |_{\Sigma_{r,\theta}}  = \xi_{\theta} |_{\Sigma_{s,\theta}}$
is Morse function away from the poles $r = \pm 1$ where $f(r)$ vanishes. Consequently, it follows from Theorem \ref{QCI} that given {\em any} two balls $B_{\pm}$ containing the poles $r= \pm 1$ respectively,
\begin{equation} \label{supboundrotation}
\sup_{M \setminus B_{\pm}} |u_h| = O(h^{-1/4}).\end{equation}
Inside $B_{\pm},$ it is well-known that there are zonal-type joint eigenfunctions that saturate the H\"{o}rmander $O(h^{-1/2})$ in an $O(h)$-neighbourhood of the poles. Consequently, one can do no better than the $\|u_h\|_{L^{\infty}(M)} = O(h^{-1/2})$ bound {\em globally} in this case.

\subsubsection{Eigenfunction decay} To verify the fold condition, we assume that $E= (1, E_2) \in {\mathcal B}_{reg}.$ From the above, we can write
\begin{equation} \label{rotationfold}
 \Lambda_{\R}(E) = \{ (r,\theta; \xi_r,\xi_{\theta} = E_2);  \xi_r^2  =  1 - f^{-2}(r)  E_2^2 \}.
 \end{equation}

Since for $E \in {\mathcal B}_{reg}$, we have $E_2^2 < \max_{r \in [-1,1]} f^2(r),$  it is clear from (\ref{rotationfold}) that the restricted projection $\pi_{\Lambda_{\R}(E)}: \Lambda_{\R}(E) \to M$ is of fold type and so the decay estimates in Theorem \ref{expdecay} are satsified. The fact that these estimates are sharp in this case, is an immediate consequence of above separation of variables and WKB estimates applied to (\ref{reduced}).\\




\subsection{Laplacians and Neumann oscillators on Liouville tori}
\subsubsection{Liouville Laplacian}
Consider the two-torus $ M=\R^2/\Z^2$   with two, smooth, positive periodic functions $a,b :\R/\Z \to \R^+$ where, for convenience, we assume that   $\min_{0 \leq x_1 \leq 1} a(x_1) > \max_{0 \leq x_2 \leq 1} b(x_2).$ The corresponding Liouville metric is given by $g = ( a(x_1) + b(x_2) ) ( dx_1^2 + dx_2^2)$ and the associated Laplacian
$$P_1(h) = - \, [ a(x_1) + b(x_2) ]^{-1}  \, ( \,  (h \partial_{x_1})^2 +  (h \partial_{x_2})^2 \, ) $$
is QCI with commutant
$$P_2(h) = -  \, [ a(x_1) + b(x_2) ]^{-1} \, (  \, b(x_2) (h\partial_{x_1
})^2 -  a(x_1)  (h\partial_{x_2})^2 \, ).$$

Given $(1,E_2) \in {\mathcal B},$ it is easily checked that
\begin{equation} \label{lagformula}
\Lambda_{1,E_2} = \{ (x_1,x_2,\xi,\eta) \in T^*(\R^2/\Z^2);  \xi^2 = E_2 + a(x_1), \,\, \eta^2 = b(x_2) - E_2 \}. \end{equation}
When $E_2 \in (\max b, \min a),$ the projection $\pi_{\Lambda_E}$ has no singularities and consequently, $\Lambda_E$ is a Lagrangian graph. On the other hand, when either $E_2 \in (\min a, \max a) \cup (\min b, \max b),$ it is easily seen from (\ref{lagformula}) that $\pi_{\Lambda_E}: \Lambda_{E} \to \R^2/\Z^2$ is of fold type. Consequently, when $a, b \in C^{\omega}(\R^2/\Z^2)$, the decay estimates in Theorem \ref{expdecay} hold for the joint eigenfunctions.

As for Theorem \ref{QCI}, we simply note that given any point $z_0 = (x_0,y_0) \in \R^2/\Z^2,$ setting $\alpha = a(x_0) > b(y_0) = \beta$ we have that
$$p_2 |_{T_{z_0}^*} = \beta (\alpha + \beta)^{-1} \xi^2 - \alpha (\alpha + \beta)^{-1} \eta^2,$$
and since $S_{z_0}^* = \{ (\xi,\eta); \xi^2 + \eta^2 = \alpha + \beta >0 \}.$  the Morse property of $p_2 |_{S_{z_0}^*}$ follows since $\alpha > \beta.$ Indeed, in terms of the parametrization $ [0,2\pi] \ni \theta \mapsto (\sqrt{\alpha + \beta} \cos \theta, \sqrt{\alpha + \beta} \sin \theta),$ the function  $p_2 |_{S_{z_0}^*}(\theta) = \beta \cos^2 \theta - \alpha \sin^2 \theta$ which is clearly Morse as a function of $\theta \in  [0,2\pi]$ when $\alpha > \beta >0.$ Consequently, the {\em global} Hardy bound 
$$ \|u_h \|_{L^{\infty}(M)} = O(h^{-1/4})$$ for joint eigenfunctions in Theorem \ref{QCI} is satisfied in this case. Moreover, it is well-known \cite{To96, TZ03} that this bound is saturated in this case.

\medskip

\subsubsection{Liouville oscillators}

In this example, the underlying Riemannian manifold is $(\R^2/\Z^2,g$ where $g$ is the above Liouville metric. Consider the Schrodinger operator
$$P_1(h) = - ( a(x_1) + b(x_2) )^{-1} \, \Big ( h^2 \partial_{x_1}^2 + h^2 \partial_{x_2}^2 \Big) + b(x_2) - a(x_1).$$
One verifies that the Schrodinger operator 
$$P_2(h) =  - ( a(x_1) + b(x_2) )^{-1} \, \Big (  b(x_2) h^2 \partial_x^2  - a(x_1) h^2 \partial_{x_2}^2 \Big)  - a(x_1) \, b(x_2)$$
commutes with $P_1(h)$. Given a regular value $E_1$ of $p_1,$ it is easy to check that

\begin{equation} \label{lag2}
\Lambda_{E} = \Big\{ (x_1,x_2,\xi,\eta) \in T^* \R^2/\Z^2; \,\,\begin{gathered} \xi^2 = \big( a(x_1) + E_1/2 \big)^2 + E_2 - E_1^2/4, \,\,\, \\
\eta^2 = - \big( b(x_2) - E_1/2 \big)^2 + E_1^2/4 - E_2 \,\end{gathered} \Big \}. \end{equation} \
It is clear from (\ref{lag2}) that $\pi_{\Lambda_E}$ is either regular, or has fold-type singularities.

As for the Morse condition: the same reasoning as in the case of the Liouville Laplacian shows that with $\Sigma_{E_1,z} = \{ (z,\xi); p_1(z,\xi) = E_1 \}$ the function $p_2 |_{\Sigma_{E_1,z}}$ is Morse and consequently the joint eigenfunctions satisfy the Hardy-type bounds in Theorem \ref{QCI}.\

Both the Liouville Laplacian and oscillator extend to QCI systems on tori of arbitrary dimension \cite{HW} The fold assumption is satisfied for generic joint energy levels (see also Remark \ref{nonfold} below) and so is the Morse assumption in Theorem \ref{QCI}.

\subsection{Laplacians on ellipsoids}

Consider the ellipsoid $\mc{E} = \{w \in \R^3, \sum_{j=1}^3 \frac{ w_j^2}{a_j^2} = 1 \}$ where $0< a_3 < a_2 < a_1$ are fixed constants. Then, given the rectangles $R_+:=(0,T_1) \times (0,T_2)$ and $R_{-}= (T_1, 2 T_1) \times (0, T_2)$ we let $\Phi_{\pm}: R_{\pm} \to \mc{E} \cap \{ \pm w_2 >0 \}$ be the conformal mapping sending vertices of $R_{\pm}$ to the four umbilic points $p_j; j=1,...4$ of $\mc{E}.$ We choose orientations so that $ \Phi_{\pm}$ have the property that $\Phi_{+}(x,T_2) = \Phi_{-}(2T_1 - x, T_2)$ and $\Phi_{+}(x,0) = \Phi_{-}(2T_1 - x, 0).$  We henceforth let $\Phi:= \Phi_{\pm}: R \to \mc{E}$ denote the induced conformal mapping with $\Phi |_{R_{\pm}} = \Phi_{\pm}$ and $ R:= R_{+} \cup R_{-}.$

 One can show (see~\cite{CV} ) that the intrinsic Riemannian metric on $\mc{E}$ pulled-back to $R$ is locally of Liouville form
\begin{equation} \label{ellipsoid1}
ds^2 = \big( a(x_1) + b(x_2)  \big) \, (dx_1^2 + dx_2^2), \end{equation}
where  $a$ and $b$ are certain hyperelliptic functions that extend to real-analytic function on $\R$. Moreover, $a(k T_1) = a'(kT_1) =0,$ 
$b(k T_2) = b'(kT_2)=0$ and $a''(kT_1) \neq 0, \, b''(kT_2) \neq 0$ for all $k \in \Z.$  Consequently, $ds^2$ extends to a $C^\omega$-metric on the torus
$\R^2/ \Gamma$ where $\Gamma = T_1 \Z \oplus T_2 \Z.$ Of course, the induced metric (which we continue to denote by $ds^2$) on the torus $\R^2/ \Gamma$ degenerates at the lattice points in $\Gamma.$ 

Let $T = \R^2/ 2 \Gamma,$ the torus generated by the doubled lattice $2 \Gamma$ and $\sigma: T \to T$ the natural involution given by $\sigma(z) = - z.$ Then, the automorphism $\sigma$ has precisely four fixed points given by the vertices $(0,0), (T_1,0), (0,T_2)$ and $(T_1,T_2)$ of $R_+.$ The corresponding fundamental domain is $D \subset \R^2/2 \Gamma$ where
$$D = [0,2T_1] \times [0,T_2] \, / \, \sim$$
where $(x,0) \equiv ( 2T_1 -x,0)$ and $(x,T_2) \equiv (2T_1-x,T_2).$
In view of the conformal mapping $\Phi$, this gives an identification $\mc{E} \cong T/\sigma$. Consequently, under this identification, the torus $T$ is a two-sheeted covering of the ellipsoid, $\mc{E}$ with covering map
$$ \Pi: T \rightarrow \mc{E}; \quad \Pi(z) = z^2.$$
This covering map is ramified over the umbilic points and the Riemannian metric $g$ on $\mc{E}$ has the property that
$$ ds^2 = \Pi^* g.$$ 

\subsubsection{Proof of Theorem \ref{t:ellipse}}
\begin{proof} Let $B_j; j=1,2,3,4$ be open neighbourhoods of the umbilic points $p_j; j=1,2,3,4$. Then, in the complement $\mc{E} \setminus \cup_{j} B_j$, one has local coordinates $(x,y)$ in terms of which the metric has the form (\ref{ellipsoid}). Then, the same argument as in the case of the Liouville torus using Theorem \ref{QCI} shows that for the joint eigenfunctions of the corresponding QCI system on the ellipsoid, one gets that
$$ \sup_{x \in \mc{E} \setminus \cup_j B_j} |u_h(x)| = O(h^{-1/4}).$$
On the other hand, in the neighbourhoods $B_j;j=1,..,4$ of the umbilic points, we claim that
\begin{equation} \label{localbound}
 \sup_{x \in \cup_j B_j} |u_h(x)| = O(h^{-1/2} |\log h|^{-1/2}). \end{equation}
 
 To prove (\ref{localbound}), we split the analysis into two cases:
 Case (i): Suppose first that for any fixed $\delta = 1/4 - \epsilon$ we have 
 $x \in B_j \setminus B_j(h^{\delta}).$ Using the conformal $(x_1,x_2)$ coordinates above near the umbilic point $p_j$ we have $x_1(p_j) = x_2(p_j) =0$ and  
 $$a(x_1) = C x_1^2 + O(x_1^3), \,\, b(x_2) = C' x_2^2 + O(x_2^3), \quad x = (x_1,x_2) \in B \setminus B(h^\delta).$$

 Then, since $p= (a+b)^{-1} (\xi^2 + \eta^2)$ and $q = (a+b)^{-1} (b \xi^2 - a \eta^2)$ in this case, with $\min \{ a(x_1), b(x_2) \} \gtrapprox h^{2\delta}$ when $x \in B_j \setminus B_j(h^{\delta}).$ 
 Then, 
 $$ |dq|_{S_x^*M}\big|+ \big| \, d^2 q|_{S_{x}^*M} \, \big| \geq Ch^{2\delta}, \quad \text{when} \,\, x \in B \setminus B(h^{\delta}).$$
From the stationary phase estimate in  (\ref{second}) and (\ref{upshot1}) it then follows that
 $$ |u_h(x)|^2 \leq C h^{-1} \big( h^{1/2-2\delta}  + h \big)$$

 so that
 \begin{equation} \label{outside}
 \sup_{x \in B_j \setminus B_j(h^\delta)} |u_h(x)| \leq  C_1 h^{-1/4}   h^{-\delta}  + C_2  \leq C_3 h^{-1/2 + (1/4-\delta)}.
 \end{equation}
 
 The bound in (\ref{outside}) is  quite crude, but since $0< \delta < 1/4,$ it is a polynomial improvement over the universal H\"ormander bound and more than suffices for the argument here.

 Finally, we deal with Case (ii); where $x \in B(h^{\delta}).$ To do this, consider $S^*_{p_j}\mc{E}$. We have that $p_j$ is self-conjugate with constant return time $T_0>0$. There is a hyperbolic source/sink pair $\xi^{\pm}\in S^*_{p_j}\mc{E}$. In particular, let $U^{\pm}\subset S^*_{p_j}\mc{E}$ be neighborhoods of $\xi^{\pm}$. Then there is $C_{U_{\pm}}$ so that for $\xi\in S^*_{p_j}\mc{E}\setminus U^{\pm}$, 
 $$
 d(G^{nT_0}(p_j,\xi),\xi^{\mp})\leq C_{U_{\pm}}e^{-|n|/C_{U_{\pm}}},\qquad \mp n\geq 0.
 $$
Moreover, we have
 $$
 |dG^t|_{TS_{p_j}^*\setminus U_{\pm}}|\leq C_{U_{\pm}}e^{-|t|/C_{U_{\pm}}},\qquad \mp t\geq 0.
 $$
 Therefore, applying~\cite[Lemmas 5.1,5.2]{CG18} to both $A_{\pm}:=S_{p_j}^*\setminus U_{\pm}$, we have, using~\cite[Theorem 5]{CG18},
 \begin{equation}
 \label{case2}
 \sup_{x\in B_j(h^\delta)}|u_h(x)|\leq Ch^{-\frac{1}{2}}|\log h|^{-1/2}.
 \end{equation}
   
%
%
%

In summary, from (\ref{outside}) and (\ref{case2}) it follows that for joint eigenfunctions on the ellipsoid, one gets the {\em global} sup bound
\begin{equation*} \label{ellipsoid}
 \|u_h\|_{L^\infty(\mc{E})} = O(h^{-1/2} |\log h|^{-1/2})
 \end{equation*}
 which proves Theorem~\ref{t:ellipse}. \end{proof}

\appendix

\section{Uniformity of Parametrix Construction}

Since the purpose of this section is to understand uniformity in $\mu$, we will write $T_\Lambda=T_{\Lambda,\mu}$. 

\begin{prop}
\label{p:unifElliptic}
Suppose that $P\in S^{0,k}_{cla}$ a classically analytic pseuodifferential operator with $|p(\alpha)|\geq c\langle \xi\rangle^k$ on $|\alpha_\xi|\geq K$, $\alpha \in \Lambda$. There is $\mu_0>0$ and $C>0$ so that for $\mu>\mu_0$ there is $h_0=h_0(\mu)$ so that for all $0<h<h_0$ and $u\in L^2$ with $Pu=0$, 
$$
\| T_{\Lambda,\mu} u\|_{L^2(|\alpha_\xi|\geq K)}\leq Ce^{-1/Ch}\|u\|_{L^2}.
$$
\end{prop}
\begin{proof}

Let $\psi\in C_c^\infty(\Lambda\cap \{|\alpha_\xi|<k\})$ so that $|p|\geq \frac{c}{2}\langle \xi\rangle^k$ on $\supp (1-\psi)$.
First note that, 
$$ 
 T_{\Lambda,\mu}u(\alpha_x,\mu \alpha_\xi)=\int_M e^{\frac{i}{\tilde{h}}[ \exp_y^{-1}(\alpha_x)\cdot \beta_\xi+\frac{i}{2}r^2(\alpha_x,y)\langle\beta_\xi\rangle]}a(\alpha_x,\mu\alpha_\xi,y)\chi(r(\alpha_x,y))u(y)dy 
 $$
 with $\tilde{h}=h/\mu$. 
By a standard application of analytic stationary phase 
$$
(1-\psi(\alpha_x,\mu \alpha_\xi ))(T_{\Lambda,\mu}Pu)(\alpha_x,\mu \alpha_\xi)=(1-\psi(\alpha_x,\mu \alpha_\xi))(T_{q,\Lambda,\mu}u)(\alpha_x,\mu\alpha_\xi) +R_{\Lambda,\mu}u
$$
where 
$$
T_{q,\Lambda,\mu}u(\alpha_x,\mu \alpha_\xi)=\int_M e^{\frac{i}{\tilde{h}}[ \exp_y^{-1}(\alpha_x)\cdot \alpha_\xi+\frac{i}{2}r^2(\alpha_x,y)\langle\alpha_\xi\rangle]}a(\alpha_x,\mu\alpha_\xi,y)q(\alpha_x,\alpha_\xi,y;\mu,h)\chi(r(\alpha_x,y))u(y)dy 
$$
with
$$
q(\alpha,y)=\sum_{j=0}^{C^{-1}\langle \alpha_\xi\rangle \tilde{h}^{-1}}\tilde{p}_j(y,-\mu d_y\varphi(\alpha, y))\mu^j\tilde{h}^j,\qquad \tilde{p}_j\in S_{cla}^{0,k-j},\qquad \tilde{p}_0=p_0,
$$
 $\varphi=\exp_y^{-1}(\alpha_x)\cdot \alpha_\xi +\frac{i}{2}r^2(\alpha_x,y)\langle \alpha_\xi\rangle,$
and $R_{\Lambda,\mu}u=O(e^{-\langle \mu \alpha_\xi\rangle/Ch}\|u\|_{L^2}).$  Here, the remainder bound comes from the fact that we have 
$$
|\tilde{p}_j(y,-\mu d_y\varphi(\alpha,y)|\leq C^j j!\langle \mu \alpha_\xi\rangle^{m-j}
$$
Observe also that since $d_y\varphi = -\alpha_\xi +O(r(\alpha_x,y))$, and $r(\alpha_x,y)ll1$, we have that $p_0(y,-\mu d_y\varphi)$ is elliptic on $\supp (1-\psi(\alpha_x,\mu \alpha_\xi)).$

Next, since $Pu=0$, we have that 
$$
(1-\psi(\alpha))T_{q,\Lambda,\mu}u(\alpha)=O(e^{-\langle \alpha_\xi\rangle/Ch}\|u\|_{L^2}). 
$$
Therefore, we need only show that one can replace $T_{q,\Lambda,\mu}$ by $T_{\Lambda, \mu}$. For this, we follow the construction in~\cite[Propoosition 6.2]{Sj} (see also~\cite[Proposition 2.2]{GalkToth}). As above, when it comes to the application of stationary phase, we rescale $\alpha_\xi \mapsto \mu \alpha_\xi$ and the small parameter is $\tilde{h}=h/\mu$, but derivatives of the symbol acquire powers of $\mu$. The same arguemnts then complete the proof.
\end{proof}

\begin{prop}
With $T_{\Lambda,\mu}$ as above, there exists $\mu_0>0$, so that for all $N>0$ there is $C_N>0$ so that for all $\mu>\mu_0$ there is $h_0(\mu)$ so that for $0<h<h_0$,
$$
S_{\Lambda,\mu}T_{\Lambda,\mu}=\Id +R_\mu
$$
where 
$$
\|R_\mu\|_{L^2\to C^N}\leq C_Ne^{-1/(hC_N)}.
$$
\end{prop}
 \begin{proof} After  rescaling the fiber coordinates $\alpha_{\xi} \mapsto \mu \alpha_{\xi}$ and setting $\tilde{h}:= \frac{h}{\mu}$, we have 
 $$T_{\Lambda}u(\alpha_x,\mu\alpha_\xi)=\int_M e^{\frac{i}{\tilde{h}}[ \exp_y^{-1}(\alpha_x)\cdot \alpha_\xi+\frac{i}{2}r^2(\alpha_x,y)\langle\alpha_\xi\rangle]}a(\alpha_x,\mu\alpha_\xi,y)u(y)dy 
 $$ it follows 
 by the standard  left parametrix construction  for $T_{\Lambda}(h)$ the one can find a formal analytic symbol $b \sim \sum_j b_j  h^j$ and associated left parametrix as in (\ref{left}) with the property that
 $$ S_{\Lambda}(\tilde{h}) T_{\Lambda}(\tilde{h}) = Id + R_{\mu}(\tilde{h})$$
 where $$ \| R_{\mu}(\tilde{h}) \|_{C^{\infty}} = O( e^{-C(\mu)/\tilde{h}}).$$
 
 An explicit realization of $b$ is of the form
$$b_{\mu}(\alpha;h) = \sum_{j; |j| \leq \tilde{h} /C_1}   b_j(\alpha;\mu)$$
and it is not difficult to show that by standard Cauchy estimates 
\begin{equation} \label{Cauchyestimates}
 |b_j(\alpha;\mu)| \leq C_0 C^j \, j! \, \mu^{j} \tilde{h}^j \langle \alpha_\xi\rangle^{-j} = C_0 C^j \, j! \, h^j \langle \alpha_\xi\rangle^{-j}.
 \end{equation}
The extra $\mu^j$ factor in (\ref{Cauchyestimates}) comes from the rescaling $\alpha_{\xi} \mapsto \mu \alpha_{\xi}$ and the parametrix construction above (note that each $\alpha_{\xi}$-derivative of the rescaled symbols pulls out a factor of $\mu$). Using (\ref{Cauchyestimates}) and Stirling's formula it then follows that for $\mu \geq \mu_0$ there is a uniform constant $C>0$ such that
$$ \| R_{\mu}(\tilde{h}) \|_{C^{\infty}} = O( e^{-C/h}).$$
That proves the Proposition and establishes the uniform bound we need in (\ref{inversion}). \end{proof}
\bibliography{biblio}
\bibliographystyle{alpha}

\end{document}